\documentclass[10pt]{amsart} 
\usepackage[utf8]{inputenc}

\usepackage{amssymb}
\usepackage{amsmath}
\usepackage{amsthm}
\allowdisplaybreaks 

\usepackage[left=2cm, right=2cm, top=3cm, bottom=3cm]{geometry}
\usepackage{color} 
\usepackage{comment} 

\usepackage{graphicx}
\graphicspath{{figures/}}
\usepackage{enumitem}

\usepackage{xcolor}
\definecolor{czerwony}{RGB}{237,28,36}
\definecolor{zielony}{RGB}{11,161,75}
\definecolor{niebieski}{RGB}{33,63,154}

\usepackage{array}

\newtheorem{theorem}{Theorem}[section]
\newtheorem*{theorem*}{Theorem}

\newtheorem{definition}[theorem]{Definition}
\newtheorem{lemma}[theorem]{Lemma}
\newtheorem{remark}[theorem]{Remark}

\newtheorem{example}[theorem]{Example}

\let\Im\relax
\DeclareMathOperator{\Im}{Im}
\let\Re\relax
\DeclareMathOperator{\Re}{Re}
\DeclareMathOperator{\avg}{avg}

\DeclareMathOperator{\nint}{nint}

\newcommand{\ser}{\mathtt{ser}}

\newcommand{\Brjuno}{\mathtt{Brjuno}}
\newcommand{\SemiBrjuno}{\mathtt{SemiBrjuno}}
\newcommand{\Brj}{\mathfrak{Brj}}
\newcommand{\KL}{\mathtt{KL}}
\newcommand{\KLB}{\mathtt{KLBrj}}

\newcommand{\Strip}{\mathsf{Strip}}
\newcommand{\CritStrip}{\mathsf{CritStrip}}
\newcommand{\Away}{\mathsf{Away}}
\newcommand{\integrald}{\text{d}}
\newcommand{\Sfrak}{\Sigma}
\newcommand{\Bad}{\mathsf{Brjuno}}
\newcommand{\Good}{\mathsf{ConstType}}

\newcommand{\GammaEul}{\Gamma_\mathsf{Eul}}

\title{The one-frequency cohomological equation, Brjuno-like functions and Khintchine-L\'evy numbers}
\author{Piotr Kamieński}
\date{}

\begin{document}

\begin{abstract}
  In the paper we consider the one-frequency cohomological equation
  \begin{equation*}
   (\partial_x + \omega \partial_y) g(x,y) = a(x,y)
  \end{equation*}
  on the 2-torus with unknown $g$ and analytic initial data $a$. We identify all the frequencies $\omega$ for which the equation has an analytic solution and express the analytic solvability condition in terms of two Brjuno-like functions, providing explicit estimates on the sup-norm of $g$. As an example we estimate the Brjuno-like functions for Diophantine and Khintchine-L\'evy numbers. We also construct an example of an arbitrarily small function $a$ for which an analytic $g$ does not exist when one of the Brjuno-like functions has infinite value.
\end{abstract}

\maketitle

\section{Introduction}

In KAM theory the cohomological equation arises as a linearization of the conjugacy equation between the flow of a Hamiltonian system on an invariant torus and the constant velocity flow on a model torus. Solving it is the first step in the Kolmogorov-Newton iterative scheme \cite{kolmogorov:1954} for finding the true solution of the conjugacy equation. It is equally essential to provide good bounds on the solution $g$ in terms of the initial data $a$ and, as analyticity is lost along the way, in terms of $\delta$, the decrement of thickness of the complex domain of~$a$. Good quality of these estimates allows for the infinite number of steps in the iterative scheme to later be succesfully glued together.

The upper bound usually obtained in KAM theory is of the form $||g||_{\varrho-\delta} \leqslant \Gamma(\delta) ||a||_\varrho$, where $||\cdot||_\varrho$ denotes the sup-norm on a complex strip of radius $\varrho > 0$.\footnote{See the beginning of section \ref{sec:results} for a more precise definition.} The value of $\Gamma(\delta)$ explodes to infinity as $\delta \searrow 0$ and the type of singularity $\Gamma$~exhibits at $\delta = 0$ depends on the arithmetic properties of the number\footnote{sometimes referred to as \emph{frequency} in the paper} $\omega$ appearing in the cohomological equation. For instance the classical results of R\"ussmann \cite{russmann:1975} give $\Gamma(\delta) = O(\delta^{-\tau})$ in the case of $\omega$ being $(C, \tau)-$Diophantine with $C>0$ and $\tau \geqslant 1$.

The main aim of the present paper is is to study both the analytic solvability of the cohomological equation and the size of $\Gamma(\delta)$ for any irrational $\omega$ in the simplest, one-frequency case of the cohomological equation being the linear PDE
\begin{equation}
   (\partial_x + \omega \partial_y) g(x,y) = a(x,y)
\end{equation}
on the 2-torus. It is known that convergence of the Brjuno function is an optimal or close-to-optimal condition in certain conjugacy problems involving small divisors \cite{brjuno:1969, yoccoz:linearyzacja-kwadratowego, russmann:brjuno-paper, buff-cheritat}. We introduce what we call the \emph{semi-Brjuno condition} which, in a sense, provides an ``if and only if'' statement about the analytic solvability of the cohomological equation. This condition is given as a convergence requirement for certain series stemming from the continued fraction expansion of $\omega$, which we label \emph{Brjuno-like functions}. The estimates on $||g||_{\varrho - \delta}$ in terms of these Brjuno-like functions are also given.

It turns out that semi-Brjuno numbers are a non-trivial extension of the classical Brjuno numbers. This, along with the known optimal conjugacy results, provides certain insight into these conjugacy problems which involve small divisors and utilize the Kolmogorov-Newton iterative scheme. Namely it serves as evidence that the ``gluing of the infinite number of steps'' part in the scheme requires a more strict arithmetic condition on $\omega$ than the ``linearized'' part of passing one step further.

To demonstrate that our estimates on $||g||_{\varrho - \delta}$ are close to optimal we compute upper bounds on the Brjuno-like functions for Diophantine frequencies $\omega$. The estimates differ from R\"ussmann's bounds by an $O(\log^2(\delta^{-1}))$ factor, however we conjecture that going down to at least $O(\delta^{-\tau} \log(\delta^{-1}))$ seems possible. To provide evidence that the convergence of the Brjuno-like functions is necessary for the existence of an analytic solution $g$ we construct, for a fixed non-semi-Brjuno $\omega$, arbitrarily small initial data $\hat a$ for the cohomological equation, for which its solution $\hat g$ is not analytic in some complex domain that is strictly smaller than that of $\hat a$.

Additionally we study what we call the \emph{Khintchine-L\'evy numbers}, introduced in a parallel paper \cite{kamienski:2018-khintchine-levy}. These numbers are somewhat similar in nature to the measure-theoretically negligible constant type numbers\footnote{i.e. Diophantine with the lowest reasonable exponent $\tau = 1$ (see definition \ref{def:diophantine}) or equivalently ones with a bounded sequence of partial quotients}, but form a set of almost full measure. Inspired by the classical Khintchine and L\'evy constants theorems \cite{khintchine:khintchine-constant-original-paper, levy:khintchine-levy} the Khintchine-L\'evy numbers are defined as those for which the sequence of products of partial quotients $M_n = a_1 \ldots a_n$ meets estimates that closely resemble the asymptotic growth rate predicted by Khintchine's theorem: $e^{(\kappa - T)n}~\leqslant~M_n~\leqslant~e^{(\kappa + T)n}$ for all $n \geqslant N$, where $e^\kappa \approx 2.685\ldots$ is the Khintchine constant, while $T > 0$ and $N \in \mathbb{N}$ are fixed parameters. In \cite{kamienski:2018-khintchine-levy} estimates on the measure of the set of these numbers are given and in the present paper we prove that $\Gamma(\delta) = O(\delta^{-\gamma} \log(\delta^{-1}))$ with $\gamma \approx 1.428\ldots$ in their case\footnote{Actually it seems plausible that one could use $1$ instead of $\gamma$ in the expression for $\Gamma(\delta)$ if the definition of Khintchine-L\'evy numbers were to be slightly altered, however for reasons laid out in \cite{kamienski:2018-khintchine-levy} the measure they occupied would be more difficult to control.}.

An important feature of Khintchine-L\'evy numbers is the insensitivity of the involved parameters to changes of~$\omega$ on the initial partial quotients and thus only a mild sensitivity of $\Gamma(\delta)$ to such changes. This stems from the requirement that the growth rate of $M_n$ is controlled starting only at some $n = N$. In the Diophantine case such modifications of $\omega$ retain the exponent $\tau$, but might change $C$ by a multiplicative factor.\footnote{Consider for instance the case of $\varphi$ whose partial quotients are all equal to $1$ and a noble number $\varphi_A$ constructed by replacing the first $1$ in $\varphi$ with a fixed, moderately large $A \in \mathbb{N}$. In this case $C_{\tau = 1}(\varphi_A) \approx O(A^{-1}) C_{\tau = 1}(\varphi)$, where $C_{\tau=1}(\cdot)$ denotes the best possible Diophantine $C$ with fixed $\tau = 1$: $C_{\tau = 1}(\omega) := \inf |q||q\omega - p|$ with the infimum taken over all nonzero integer pairs $(p,q)$.} Plugging that into R\"ussmann's estimates gives a multiplicative correction in $\Gamma(\delta)$ since $\Gamma(\delta) = O(C^{-1})$ in terms of $C$, which in turn affects the applicability thresholds of KAM theorems, since these thresholds are usually polynomial in $C$ (e.g. \cite{delallave:kam-without}). For Khintchine-L\'evy numbers, however, a change on the initial partial quotients has no effect on the arithmetic properties of the frequency and, as we prove in theorem \ref{thm:brjuno-estimates}, $\Gamma(\delta)$ suffers only a minor additive $O(1)$ correction in this case. We conjecture that this milder sensitivity of $\Gamma(\delta)$ should also entail a milder sensitivity of the applicability thresholds of KAM theorems.

The remaining part of this paper is structured as follows. In section \ref{sec:results} we first set the stage by defining all the necessary tools of the trade. We then formally recall/introduce Diophantine numbers, Brjuno numbers, semi-Brjuno functions and numbers as well as Khintchine-L\'evy numbers. We conclude section \ref{sec:results} by stating the existence and upper bound results in the most general semi-Brjuno case and the specific Diophantine and Khintchine-L\'evy cases. Section \ref{sec:ideas} contains a heuristic overview of the ideas used in the proofs of the theorems (stripped of technical details as much as possible), while in section \ref{sec:proofs} we fill in all the necessary technicalites and present the proofs of the theorems in full. In section~\ref{sec:brjuno-semibrjuno} we prove that Brjuno numbers are also semi-Brjuno and construct an example that demonstrates the lack of reverse inclusion for these two sets. We conclude the paper with the construction of $\hat a$, the initial data that gives rise to a non-analytic solution $\hat g$ given a non-semi-Brjuno frequency, in section \ref{sec:counterexample}.

\section{Statement of results}\label{sec:results}

\subsection{Notations and basic notions}

Before we proceed with the formulation of our main results we introduce a few terms that will be essential for us. We first specify what we mean precisely by ``analyticity''. In our setting functions are analytic if they can be extended from $\mathbb{R}$ to its complex neighborhood to form holomorphic functions.

For $\varrho > 0$ we set $\Pi(\varrho) = \{ z \in \mathbb{C} : |\Im z| < \varrho \}$ and define $\mathcal{P}_0(\varrho)$ to be the set of all continuous functions $b : \overline{\Pi(\varrho)^2} \mapsto \mathbb{C}$ which are holomorphic on $\Pi(\varrho)^2$, $2\pi$-periodic in both variables and for which $b(\mathbb{R}^2) \subset \mathbb{R}$. We additionally require their averages to vanish:
\begin{equation}
  \avg b := {1 \over (2\pi)^2} \iint_{[0,2\pi)^2} b(x,y) \, \integrald x \, \integrald y = 0.
\end{equation}
The space $\mathcal{P}_0(\varrho)$ is equipped with the standard supremum norm, denoted $|| \cdot ||_\varrho$.

We will be working with irrational frequencies $\omega$, for convenience from the $(0,1)$ interval\footnote{The unit interval restriction is not necessary, we consider it only to avoid dealing with the zeroth partial quotient $a_0$.}: $\omega \in \mathbb{X} := (0,1) \setminus \mathbb{Q}$. All numbers in $\mathbb{X}$ have a unique continued fraction expansion of the form
\begin{equation}
 \omega = [a_1, a_2, a_3, \ldots] = \cfrac{1}{a_1
              + \cfrac{1}{a_2
              + \cfrac{1}{a_3
              + \cfrac{1}{\ddots}}}}
\end{equation}
given by a sequence of positive integers (or \emph{partial quotients} to $\omega$) $a_j, j \geqslant 1$. By \emph{convergents} to $\omega$ we mean the reduced fractions $p_n / q_n$ obtained by truncating the continued fraction expansion of $\omega$ at $a_n$.

Whenever an asterisk in the subscript is used with regard to a set of summation indices we understand it as the exclusion of the index $0$: $S_* := S \setminus \{ 0 \}$. By $\GammaEul$ and $\GammaEul'$ we will denote the Euler Gamma function and its derivative. The Greek letter $\varphi$ is reserved for the golden ratio: $\varphi = (1 + \sqrt{5})/2$. For $\alpha \in \mathbb{R}$ we denote the nearest integer to $\alpha$ by $\nint(\alpha)$.

We now define the four classes of numbers that will be important for us: the Diophantine, Brjuno, semi-Brjuno and Khintchine-L\'evy numbers.

\begin{definition}[Diophantine numbers]\label{def:diophantine}
 Let $\tau \geqslant 1$ and $C > 0$. We say that a real number $\omega$ is \emph{$(C,\tau)$-diophantine} if the inequality
 \begin{equation}\label{eq:dioph-condition}
  |q\omega - p| \geqslant {C \over |q|^\tau}
 \end{equation}
 holds for all integers $p$ and $q$ with $q \neq 0$. A number is called \emph{Diophantine} if it is $(C, \tau)$-Diophantine for some $C > 0$ and $\tau \geqslant 1$.
\end{definition}

\begin{definition}[Brjuno numbers]
 A number $\omega \in \mathbb{X}$ is a \emph{Brjuno number} whenever
 \begin{equation}
  \sum_{n=1}^\infty {\log q_{n+1} \over q_n} < \infty.
 \end{equation}
 We denote the set of Brjuno numbers as $\Brjuno$.
\end{definition}

\begin{definition}[Brjuno-like functions and semi-Brjuno numbers]
 Let $\Delta > 0$. We define the Brjuno-like functions $\Brj_1(\Delta), \Brj_2(\Delta)$ and $\Brj(\Delta)$ through
 \begin{align}\label{eq:brjuno-functions-definition}
 \begin{split}
  \Brj_1(\Delta) = \sum_{n=1}^\infty e^{-q_n \Delta} q_{n+1}, \qquad \Brj_2(\Delta) = \sum_{n=1}^\infty e^{-q_n \Delta} q_{n+1} \log a_{n+1}, \qquad \Brj(\Delta) = 2\Brj_1(\Delta) + \Brj_2(2\Delta).
 \end{split}
 \end{align}
 A number $\omega \in \mathbb{X}$ is called \emph{$\Delta$-semi-Brjuno} when both $\Brj_1(\Delta)$ and $\Brj_2(2\Delta)$ are convergent series. If they converge for all\footnote{contrary to ``for some'' in the Diophantine case} $\Delta > 0$, then $\omega$ is a \emph{semi-Brjuno number}. We denote the sets of $\Delta$-semi-Brjuno numbers and semi-Brjuno numbers by $\SemiBrjuno(\Delta)$ and $\SemiBrjuno$, respectively.
\end{definition}

For a fixed $\omega \in \mathbb{X}$ we set $a_j'(\omega) = 1 + a_j(\omega)$ (here $j = 1, 2, \ldots$). We also define two sequences $(M_n)_{n=1}^\infty$ and $(M_n')_{n=1}^\infty$ through $M_n^\dag = a_1^\dag \ldots a_n^\dag$, where the symbol $\dag$ indicates that equality holds both when an empty superscript and a prime $'$ is used. We will adopt this notation from now on. This way we also define two universal constants $\kappa$ and $\kappa'$ through $\kappa^\dag = \int_\mathbb{X} \log a_1^\dag (x) \, \integrald \gamma (x)$, where $\gamma$ is the Gauss measure on $\mathbb{X}$: $\integrald \gamma(x) = ((1+x)\log 2)^{-1} \integrald x$. Note that $e^\kappa$ is actually the classical Khintchine constant - the Gauss map $G(\omega) = \omega^{-1} - \lfloor \omega^{-1} \rfloor$ is ergodic (\cite{ryll-nardzewski}) and thus by Birkhoff's pointwise ergodic theorem the temporal ${1 \over n} \log M_n$ averages of $\log a_1$ tend to its spatial average almost everywhere. Similarly $\kappa'$ can be thought of as the counterpart of the Khintchine constant for the sequence $(M_n')$.

\begin{definition}[Khintchine-L\'evy numbers ($\KL$-numbers)]\label{def:KLcondition}
 We say that an irrational number $\omega$ is \emph{upper-$\KL^\dag$} with constants $T_+ > 0$ and $N \in \mathbb{N}$ if the following inequality holds for all $n \geqslant N$:
 \begin{equation}
  M_n^\dag(\omega) \leqslant e^{(\kappa^\dag + T_+)n}.
 \end{equation}
 Similarily, a number is \emph{lower-$\KL^\dag$} with constants $T_- > 0$ and $N \in \mathbb{N}$ if for all $n \geqslant N$ we have
 \begin{equation}
   e^{(\kappa^\dag - T_-)n} \leqslant M_n^\dag(\omega).
 \end{equation}

 We denote the sets formed by the numbers $\omega$ with the above properties by, respectively, $\KL^{\dag+}(T_+,N)$ and $\KL^{\dag-}(T_-,N)$. By $\KL^\dag(T_-, T_+, N)$ we understand the intersection of $\KL^{\dag+}(T_+, N)$ and $\KL^{\dag-}(T_-, N)$ and we will write $\KL^\dag(T,N)$ short for $\KL^\dag(T,T,N)$ where $T > 0$.
\end{definition}

Brjuno and semi-Brjuno numbers form a full measure set, while Diophantine numbers with fixed parameters form a set, whose measure is close to full. In \cite{kamienski:2018-khintchine-levy} we prove this is also the case for $\KL$-numbers:
\begin{theorem}[\cite{kamienski:2018-khintchine-levy}]
 Suppose that $T > 0$ and $N \in \mathbb{N}$. Then the Gauss measure of $\KL(T,N)$ satisfies
 \begin{equation}
  \gamma(\KL(T,N)) \geqslant 1 - O\left(\sqrt{N} \cdot \Xi(T)^{\sqrt{N}}\right)
 \end{equation}
 with explicitly computed $0 < \Xi(T) < 1$ and constants in $O(\ldots)$. In particular for any fixed $T$ there is a large enough $N$ that brings $\gamma(\KL(T,N))$ arbitrarily close to $1$.  
\end{theorem}

\subsection{Results in the semi-Brjuno case}

\begin{theorem}\label{thm:general-estimate}
 Suppose $\omega \in \mathbb{X}, \varrho > 0$ and $a \in \mathcal{P}_0(\varrho)$. Let $\delta \in (0, \varrho)$ and put $\Delta := (1+\omega)\delta$. The solution $g$ of the cohomological equation
 \begin{equation}
  (\partial_x + \omega \partial_y) g(x,y) = a(x,y)
 \end{equation}
 exists and satisfies $g \in \mathcal{P}_0(\varrho - \delta)$ provided that $\omega \in \SemiBrjuno(\Delta)$. Moreover, the estimate on $||g||_{\varrho - \delta}$ is given by
 \begin{equation}
  ||g||_{\varrho - \delta} \leqslant \Gamma(\delta) \cdot ||a||_\varrho,
 \end{equation}
 where $\Gamma(\delta)$ is given by
 \begin{equation}\label{eq:Gamma-of-delta}
  \Gamma(\delta) = (2 + O(\delta)) \cdot \Brj(\Delta) + G_\mathsf{ConstType} \cdot \delta^{-2} + G_\mathsf{Away} \cdot \delta^{-1} \log (\delta^{-1})
 \end{equation}
 and $G_\mathsf{Away}$ and $G_\mathsf{ConstType}$ are absolute constants with $O(\delta)$ increments given by \eqref{eq:GAway} and \eqref{eq:GGood}, respectively.
\end{theorem}

\subsection{Results in the Diophantine case}

Since theorem \ref{thm:general-estimate} provides an upper bound on the norm of the solution of the cohomological equation mainly in terms of $\Brj_1(\Delta)$ and $\Brj_2(2\Delta)$ we focus on estimating the Brjuno-like functions.

\begin{theorem}\label{thm:brj12-for-diophantine}
 Suppose $\omega \in \mathbb{X}$ is $(C, \tau)-$Diophantine with $C > 0$ and $\tau \geqslant 1$. Let $\delta > 0$ and put $\Delta = (1+\omega)\delta$. Suppose $\Delta$ is small enough, that is $\Delta \leqslant \min \{ \tau^{-1}, \tau e^{-1} \}$. Then the Brjuno-like functions $\Brj_1(\Delta)$ and $\Brj_2(\Delta)$ can be estimated by
 \begin{equation}\label{eq:Brj1-estimate-diophantine}
  \Brj_1(\Delta) \leqslant {C^{-1} (\tau e^{-1})^\tau \over \log \varphi} \cdot \Delta^{-\tau} \cdot P_1( \log(\Delta^{-1}) )
 \end{equation}
 and
 \begin{equation}\label{eq:Brj2-estimate-diophantine}
  \Brj_2(\Delta) \leqslant {C^{-1} (\tau - 1) (\tau e^{-1})^\tau \over \log \varphi} \cdot \Delta^{-\tau} \cdot P_2( \log(\Delta^{-1}) ),
 \end{equation}
 where $P_1$ and $P_2$ are polynomials of degree $1$ and $2$, respectively, with leading coefficients equal to $1$:
 \begin{equation}
  P_1(X) = X + G_1^{(0)}, \qquad P_2(X) = X^2 + G_2^{(1)} X + G_2^{(0)}.
 \end{equation}
 The lower order coefficients are given by
 \begin{align}
  \begin{split}
   G_1^{(0)} &= \log(3\tau\varphi) + {\GammaEul(\tau) \over 2 (\tau e^{-1})^\tau}, \qquad G_2^{(1)} = {C \log (C^{-1}) \over \tau - 1} + {1 \over 2} {\GammaEul(\tau) \over (\tau e^{-1})^\tau} + \log(3\varphi(\tau+1)^2),\\
   G_2^{(0)} &= {C \log (C^{-1}) \over \tau - 1} \cdot \left( \log(3\varphi\tau) + {\GammaEul(\tau) \over 2 (\tau e^{-1})^\tau} \right) + \log(3\varphi(\tau+1)) \log(\tau+1) + {\GammaEul'(\tau) \over (\tau e^{-1})^\tau}.
  \end{split}
 \end{align}
\end{theorem}

Note that in the special case of $\tau = 1$ the estimates for $\Brj_2(\Delta)$ are not properly expressed when written in the form \eqref{eq:Brj2-estimate-diophantine} as $\tau - 1$ is somewhat artificially taken out to stand in front of $P_2(\log(\Delta^{-1}))$ and there appear divisions by $\tau - 1$ in $G_2^{(1)}$ and $G_2^{(0)}$. In this case one should include the $\tau - 1$ factor in $P_2$ so that it cancels out with the denominators in the coefficients. As a result $P_2$ will become a polynomial of degree $1$: $P_2(X) = C \log (C^{-1}) \cdot X + \log(3\varphi) + e/2$. This, after cancellations, makes the leading constant (i.e. the one at $\Delta^{-\tau} \log (\Delta^{-1})$) equal to $e^{-1} \log (C^{-1}) / \log \varphi$.

\subsection{Results in the Khintchine-L\'evy case}

We provide estimates for $\Brj_1(\Delta)$ and $\Brj_2(\Delta)$ whenever
\begin{equation}
 \omega \in \KLB(T_-, T_+, N) := \KL'^+(T_+,N) \cap \KL^-(T_-,N).
\end{equation}
We choose this set because of the form of the addends in $\Brj_1$ and $\Brj_2$ and the fact that $M_n < q_n < M_n'$ for all $n$. Indeed, we will require estimates on $q_n$ from below because of the $e^{-q_n\Delta}$ term and from above because of the $q_{n+1}$ term, which are provided by $\KL^-$ and $\KL'^+$, respectively. Similarly the term $\log a_{n+1} = \log (M_{n+1} / M_n)$ will use the estimate on $M_{n+1}$ from above and on $M_n$ from below.\footnote{Since $M_{n+1} < M_{n+1}'$ we do not additionaly assume that $\omega \in \KL^+(T)$ for some $T$ and use $\omega \in \KL'^+(T_+)$ instead.} 

To make the notations more compact we denote
\begin{equation}
 \beta = \kappa - T_-, \qquad \beta' = \kappa' + T_+, \qquad \gamma = \beta'/\beta.
\end{equation}

\begin{theorem}[Estimates of the Brjuno-like functions in the particular case of $\KL$ frequencies]\label{thm:brjuno-estimates}
 Suppose that $\omega \in \KLB(T_-, T_+, N)$ with some $T_-, T_+ > 0$ and $N \in \mathbb{N}$ and let $\Delta > 0$. The sizes of $\Brj_1(\Delta)$ and $\Brj_2(\Delta)$ satisfy
 \begin{align}\label{eq:Brj-estimates-for-omega-KLBrj}
 \begin{split}
  \Brj_1(\Delta) &\leqslant G_{\KLB1} \cdot \Delta^{-\gamma} + O(1) \\
  \Brj_2(\Delta) &\leqslant G_{\KLB21} \cdot \Delta^{-\gamma} + G_{\KLB22} \cdot \Delta^{-\gamma} \log \left( \Delta^{-1} \right) + O(1)
 \end{split}
 \end{align}
 where the constants in the inequalities are given by $G_{\KLB1} = e^{\beta'} ( \beta^{-1} \GammaEul(\gamma) + (\gamma e^{-1} )^\gamma ), G_{\KLB22} = (T_+ + T_-) \cdot e^{\beta '} ( ( \gamma e^{-1} )^\gamma + \beta^{-2} \GammaEul(\gamma) )$ and $G_{\KLB21} = e^{\beta'} ((T_+ + T_-) \cdot ( \beta^{-1} \GammaEul(\gamma) + (\gamma e^{-1})^\gamma) + (\gamma e^{-1} )^\gamma \log(2\gamma) + \beta^{-1} \GammaEul'(\gamma) )$.
 
 The $O(1)$ terms depend only on the $N-1$ initial partial quotients of $\omega$ and are given by formulae \eqref{eq:brjfindiff} with $m = N$.
\end{theorem}

\begin{remark}[Numerical values of constants of theorem \ref{thm:brjuno-estimates} for certain parameter values]
 Since we require a lower bound on $q_n$ and this sequence is always bounded from below by the Fibonacci sequence $F_n$ it makes little sense to consider a lower bound on $M_n$ by a sequence smaller than $F_n$. Since $F_n \approx {1 \over \sqrt{5}} \varphi^n$ we will consider $T_-$ for which $F_n \leqslant e^{(\kappa - T_-)n}$, that is $T_- < T_-^{\mathtt{max}} = \kappa - \log \varphi \approx 0.507$.
 
 The particular approximate numerical values of the constants involved in the estimates of theorem \ref{thm:brjuno-estimates} are given in table \ref{table:klbrj-numerical}.
 \begin{table}[h]\centering
  \begin{tabular}[h]{@{\phantom{.}}l||m{0.95cm}||m{0.95cm}|m{0.95cm}|m{0.95cm}|m{0.95cm}||m{0.95cm}|m{0.95cm}|m{0.95cm}||m{0.95cm}|m{0.95cm}|m{0.95cm}}
  \hfill $T_- =$ & \centering $0$ & \multicolumn{4}{c||}{\centering $0.1$} & \multicolumn{3}{c||}{\centering $0.2$} & \multicolumn{3}{c}{\centering $0.5$} \\
  \hline
  \hfill $T_+ =$ & \centering $0$ & \centering $0.1$ & \centering $0.5$ & \centering $1.0$ & \centering $2.0$ & \centering $0.5$ & \centering $1.0$ & \centering $2.0$ & \centering $0.5$ & \centering $1.0$ & \hspace{1.2mm} $2.0$  \\
  \hline \hline
  $G_{\KLB1\phantom{1}} \approx$ & $5.3\verb'e00'$ & $6.7\verb'e00'$ & $1.2\verb'e01'$ & $3.0\verb'e01'$ & $2.8\verb'e02'$ & $1.6\verb'e01'$ & $4.6\verb'e01'$ & $5.8\verb'e02'$ & $1.0\verb'e02'$ & $7.1\verb'e02'$ & $6.6\verb'e04'$ \\
  \hline
  $G_{\KLB21} \approx$  & $1.6\verb'e00'$ & $4.8\verb'e00'$ & $1.8\verb'e01'$ & $6.8\verb'e01'$ & $1.0\verb'e03'$ & $2.6\verb'e01'$ & $1.1\verb'e02'$ & $2.2\verb'e03'$ & $2.5\verb'e02'$ & $2.3\verb'e03'$ & $3.1\verb'e05'$  \\
  \hline
  $G_{\KLB22} \approx$  & $0.0\phantom{e00'}$ & $6.2\verb'e-1'$ & $5.2\verb'e00'$ & $3.2\verb'e01'$ & $7.2\verb'e02'$ & $1.0\verb'e01'$ & $6.2\verb'e01'$ & $1.8\verb'e03'$ & $2.2\verb'e02'$ & $2.6\verb'e03'$ & $4.8\verb'e05'$  \\
  \end{tabular}
  \caption{Approximate numerical values of $G_{\KLB1}, G_{\KLB21}$ and $G_{\KLB22}$ for particular $T_-$ and $T_+$ (written in scientific notation)}\label{table:klbrj-numerical}
 \end{table}
\end{remark}

\section{Ideas and heuristics}\label{sec:ideas}

The general estimate on $||g||_{\varrho - \delta}$ is given in the following
\begin{lemma}[Convergence of the formal solution]\label{lem:cohomological-estimates-norm-times-series}
  Suppose that $\varrho > 0$, $\delta \in (0, \varrho)$ and that $a \in \mathcal{P}_0(\varrho)$. The cohomological equation with initial data $a$, i.e.
  \begin{equation}
   (\partial_x + \omega \partial_y) g(x,y) = a(x,y)
  \end{equation}
  has a unique solution within the class of zero-mean formal Fourier series\footnote{i.e. not necessarily convergent Fourier series with $g_{0,0} = 0$}.
  
  If additionally the infinite series $\ser(\delta)$ given by
  \begin{equation}
    \ser(\delta) = \sum_{(p,q)\in\mathbb{Z}^2_*} L(q,p)
  \end{equation}
  where each summand $L(q,p)$ is defined as
  \begin{equation}
    L(q,p) = {e^{-(|p|+|q|)\delta} \over |q\omega - p|}
  \end{equation}
  converges, then this formal solution satisfies $g \in \mathcal{P}_0(\varrho - \delta)$ and the following estimate on the norm of $g$ is valid:
  \begin{equation}\label{eq:general-estimate}
    ||g||_{\varrho - \delta} \leqslant ||a||_\varrho \cdot \ser(\delta).
  \end{equation}
\end{lemma}
The exact recipe for arriving at inequality \eqref{eq:general-estimate} is given in subsections \ref{subsec:formal-solution} and \ref{subsec:preliminary-estimates}.

Note that each $L(q,p)$ depends only on the summation index $(q,p)$ and two parameters: $\omega$ and $\delta$. This makes the estimation of $||g||_{\varrho - \delta}$ a purely number-theoretic problem. 

To analyze the size of $\ser(\delta)$ we will partition the set $\mathbb{Z}_*^2$ into a family of subsets of indices and consider the series over each of these subsets. In each case we will have different lower bounds on the denominators $|q\omega - p|$ which will allow us to provide estimates for each of the ``subseries''.

We will essentially distinguish between three types of sets:
\begin{itemize}
 \item indices that are far away from the critical line $p = \omega q$,
 \item ``constant type'' indices within the critical strip $|q\omega - p| < 1$,
 \item indices related to the continued fraction expansion of $\omega$.
\end{itemize}

The first kind of indices will give us the best denominators - ones that are bounded away from zero. This will translate to $|q\omega - p|^{-1}$ being of constant order of magnitude - this way the summands $L(q,p)$ will essentially be exponential in $|q|+|p|$ and therefore their sum will closely resemble a convergent geometric series. More specifically we divide this category of indices into a countable number of strips $\Strip(n)$ indexed by $n \in \mathbb{Z} \setminus \{ -1, 0 \}$. In each of the strips we will have $n < q\omega -p < n+1$ and this way $|q\omega-p|^{-1} \approx |n|^{-1}$ and $p \approx n - q\omega$. The sum over each (half-)strip\footnote{For the purpose of this heuristic demonstration we fix the signs of $q$ and $n$, considering the sums over $q < 0$ and $n > 0$ to avoid dealing with the absolute values in the exponent.} will approximately read
\begin{equation}\label{eq:away-heuristics}
 \sum_{(q,p)\in\Strip(n), q<0} L(q,p) \approx \sum_{q < 0} {1 \over n} e^{-\delta(|q| + |q\omega - n|)} = \sum_{q < 0} {1 \over n} e^{-\delta(-q -q\omega + n)} = {e^{-n \delta} \over n} \sum_{q < 0} e^{q(1+\omega)\delta}.
\end{equation}
The final sum with fixed strip index $n$ is a geometric series of size $O(\delta^{-1})$. Summing over all $n > 0$ will yield an additional $O(\log(\delta^{-1}))$ factor since $\sum_n e^{-n\delta}/n = -\log(1 - e^{-\delta}) = O(\log(\delta^{-1}))$.

By ``constant type'' indices we mean the ones contained within the critical strip $\CritStrip = \{ |q\omega - p| < 1 \}$ which the following theorem of Legendre does not concern:
\begin{theorem}[Legendre, {\cite[Chapter II, \S 5, Theorem 1]{rockett-szusz}}]\label{thm:legendre}
 Suppose that $\omega$ is irrational and $q > 0$ and $p$ are integers. If the inequality
 \begin{equation}\label{eq:ineq-legendre}
  |q\omega - p| < {1 \over 2q}
 \end{equation}
 holds, then $p = \nint (q \omega)$ and $q = a q_k$ for some $k$ and $a$ bounded by $1 \leqslant a < \sqrt{(a_{k+1} + 2)/2}$.
\end{theorem}

We will use the upper bound on $a$ of this theorem frequently, therefore we denote 
\begin{equation}
 a^*_{k+1} = \begin{cases} \left\lfloor \sqrt{(a_{k+1} + 2)/2} \right\rfloor \mbox{ whenever } \sqrt{(a_{k+1} + 2)/2} \not\in \mathbb{N}, \\ \sqrt{(a_{k+1} + 2)/2} - 1 \mbox{ whenever } \sqrt{(a_{k+1} + 2)/2} \in \mathbb{N} \end{cases}
\end{equation}
for $k = 0, 1, 2, \ldots$.

An immediate consequence of theorem \ref{thm:legendre} is the validity of the reverse inequality to \eqref{eq:ineq-legendre} for all nonzero pairs $(q,p)$ within the critical strip, which are not of the form $(a q_k, \nint (aq_k\omega))$. This way for constant type indices we will have a linear estimate for the small denominators: $|q\omega - p|^{-1} \leqslant 2|q|$. This will again result in the sum of $L(q,p)$ being convergent, although slightly more slowly than in the previous case. The estimation will be similar to the one in \eqref{eq:away-heuristics}, the only essential difference will be a $|q|$ term instead of the $e^{-n\delta}/n$ term. This will result in the sum being $O(\delta^{-2})$ with no $\log (\delta^{-1})$ term this time since the critical strip is ``of thickness 2'' in the $p$ direction, contrary to the previous case of ``infinite thickness''. This way $O(\delta^{-2})$ is counted only twice, contrary to countably many $O(\delta^{-1})$ summands with weights decaying as $e^{-n\delta}/n$.

The last group of indices is related to the continued fraction expansion of $\omega$. These will be precisely the indices mentioned in theorem \ref{thm:legendre}. We will label them \emph{Brjuno indices} since the sum of $L(q,p)$ over these indices will be a Brjuno-like function. For general irrational $\omega$ this is the only place where the sum of $L(q,p)$ can become infinite and in order to control its size we will need additional arithimetic assumptions on the frequency $\omega$.

Within the critical strip we have $p \approx q\omega$, therefore the exponential term in $L(q,p)$ will be approximately equal to $e^{-q(1+\omega)\delta}$.\footnote{Again, for the purpose of the heuristic argument we fix the sign of $q$, this time considering $q>0$.} The $q$ we are concerned with in this case are of the form $aq_k$ with $k = 0, 1, \ldots$ and $1 \leqslant a \leqslant a^*_{k+1}$.\footnote{Note that the restriction on the size of $a$ gives $aq_k < q_{k+1}$ - this way the sum can be split into finite sums over indices $q$ contained within intervals $[q_k, q_{k+1})$.} The small denominators $|q\omega - p|$ turn into $|aq_k \omega - qp_k|^{-1} = a^{-1}|q_k \omega - p_k|^{-1} \approx a^{-1} q_{k+1}$. This way the sum we are concerned with roughly becomes
\begin{equation}
 \sum_{(q,p) \text{ of Thm. \ref{thm:legendre}}} L(q,p) = \sum_{k=0}^\infty \sum_{a=1}^{a^*_{k+1}} {1 \over a} q_{k+1} e^{-aq_k(1+\omega)\delta} 
\end{equation}
Splitting the sum into two, over terms which correspond to $a = 1$ and $a \geqslant 2$ and setting $\Delta := (1+\omega)\delta$ gives two separate ingredients: the first one being $\sum_{k=0}^\infty e^{-q_k\Delta} q_{k+1}$ and the second one roughly ${1 \over 2} \sum_{k=0}^\infty e^{-2q_k\Delta} q_{k+1} \log a_{k+1}$.\footnote{To arrive at the second one we first estimate $e^{-aq_k\Delta} \leqslant e^{-2q_k\Delta}$ and later the sum of $a^{-1}$ starting at $a=2$ by $\log a^*_{k+1} \approx (\log a_{k+1}) / 2$.}

The two series constitute the two Brjuno-like functions defined in \eqref{eq:brjuno-functions-definition} whose convergence guarantees the solvability of the cohomological equation, with the solution a member of $\mathcal{P}_0(\varrho - \delta)$. From the formulas in \eqref{eq:brjuno-functions-definition} one can readily infer that the deciding factor in the size of $\ser (\delta)$ is the growth rate of the sequence $(q_n)$. For Diophantine numbers with exponent $\tau$ this growth rate is recursively given by $q_{n+1} \leqslant O(q_n^\tau)$ (see lemma \ref{lem:diophanticity-in-terms-of-continued-fraction}) and in their case the essential steps in the estimation of the first Brjuno-like function can be summarized in the following chain of inequalities:
\begin{equation}\label{eq:diophantine-heuristic-chain}
 \sum_{k=0}^\infty e^{-q_k \Delta} q_{k+1} \leqslant C_1 \sum_{k=0}^\infty q_k^\tau e^{-q_k \Delta} \leqslant C_1 \sum_{n=1}^\infty n^\tau e^{-n\Delta} \approx C_1 \int_0^\infty x^\tau e^{-\Delta x} \integrald x = C_1 \GammaEul(\tau+1) \Delta^{-(\tau + 1)} = O(\delta^{-\tau+1}),
\end{equation}
where $C_1$ is some constant. This is consistent with the ``naive'' approach to estimating $||g||_{\varrho - \delta}$, but not so much with R\"ussmann's results \cite{russmann:1975}, where $O(\delta^{-\tau})$ is obtained. In the present paper we actually go down to $O(\Delta^{-\tau} \log (\Delta^{-1}))$ for $\Brj_1(\Delta)$ and $O(\Delta^{-\tau} \log^2 (\Delta^{-1}))$ for $\Brj_2(\Delta)$.

In the second inequality in \eqref{eq:diophantine-heuristic-chain} the gaps between $q_k$ are filled with $n$'s and we arrive at an extra $\Delta^{-1}$ factor. Instead we can use the lower bound\footnote{The estimate stems from the general lower bound on $q_k$ by the Fibonacci sequence $F_k$, which can be further estimated from below by $\varphi^k/3$.} $\varphi^k/3 < q_k$ and estimate $\mathfrak{D}_1(q_k) := q_k^\tau e^{-q_k\Delta}$ from above by $\mathfrak{D}_1(\varphi^k/3)$ as $\mathfrak{D}_1$ is decreasing on $[ \tau \Delta^{-1}, \infty)$. This strategy works for $k$ for which $\tau \Delta^{-1} < \varphi^k/3 < q_k$, i.e. $k > \log (3\tau \Delta^{-1})$. The sum over these tail $k$ gives $O(\Delta^{-\tau})$. The initial terms for $k < \log(3\tau \Delta^{-1})$ give the additional $\log ( \Delta^{-1} )$ factor.

For Khintchine-L\'evy numbers with parameters $T$ and $N$ the growth rate of $(q_n)$ is given by explicit lower and upper bounds: $e^{(\kappa - T)n} \leqslant q_n \leqslant e^{(\kappa' + T)n}$ valid for $n \geqslant N$, where $\kappa \approx 0.988$ and $\kappa' \approx 1.410$ are universal constants. In their case the essence of the estimation of the first Brjuno-like function is summarized by the following computations:
\begin{align}
\begin{split}
 \sum_{k=N}^\infty e^{-q_k \Delta} q_{k+1} &\leqslant \sum_{k=N}^\infty e^{(\kappa'+T)(n+1)} e^{-\Delta e^{(\kappa - T)n}} \approx C_2 \int_N^\infty e^{(\kappa+T)x} e^{-\Delta e^{(\kappa - T)x}} \integrald x \leqslant \\
 &\leqslant C_3 \int_0^\infty y^{(\kappa' + T)/(\kappa - T)} e^{-y\Delta} y^{-1} \integrald y = O\left(\delta^{-(\kappa' + T)/(\kappa - T)}\right).
\end{split}
\end{align}
Again, $C_2$ and $C_3$ are some constants and we used the substitution $y = e^{(\kappa - T)x}$ in between integrals.

\section{Solution of the cohomological equation}\label{sec:proofs}

In this section we formalize the ideas laid out in section \ref{sec:ideas}. We consider a fixed number $\varrho > 0$, function $a \in \mathcal{P}_0(\varrho)$ and $\delta \in (0, \varrho)$.

\subsection{Formal solution}\label{subsec:formal-solution}

The space $\mathcal{P}_0(\varrho)$ is naturally embedded in $L^2(\overline{\Pi(\varrho)^2}, \mathbb{C})$, therefore we can use Fourier expansions. We will write the Fourier basis as $e_{p,q}(x,y) = \exp(i(px - qy))$ for $(p,q) \in \mathbb{Z}^2_*$ and $x, y \in \overline{\Pi(\varrho)^2}$.\footnote{The seemingly unnatural minus sign in the definition of $e_{p,q}$ will allow us to arrive at a small denominator of the form $|q\omega - p|$ instead of $|q\omega + p|$.} In this language the vanishing of averages condition for $b \in \mathcal{P}_0(\varrho)$ translates to the vanishing of the zeroth coefficient $b_{0,0} = 0$, while the reality condition entails Hermitian symmetry of the coefficients: $\overline{b_{p,q}} = b_{-p,-q}$.

The differential operator $(\partial_x + \omega\partial_y)$ is diagonal in this basis as $(\partial_x + \omega \partial_y) e_{p,q} = i(p - q\omega)e_{p,q}$. Representing $g$ and $a$ as $g(x,y) = \sum g_{p,q} e_{p,q}(x,y)$ and $a(x,y) = \sum a_{p,q} e_{p,q}(x,y)$ with the sums running over $(p,q) \in \mathbb{Z}^2_*$ the cohomological equation can be written as an infinite system of equations on the coefficients:
\begin{equation}
 i(p - q\omega) g_{p,q} = a_{p,q}
\end{equation}
over $(p,q) \in \mathbb{Z}^2_*$. Since $\omega$ is irrational we can always retrieve $g_{p,q}$ as
\begin{equation}\label{eq:gpq}
 g_{p,q} = {a_{p,q} \over i(p - q\omega)}
\end{equation}
since $q\omega - p$ never vanishes. This is also why we are working with average-free $g$ and $a$ - this guarantees the existence and uniqueness of the formal solution, there is no ``$0/0$'' in the equation for the $(0,0)$-th mode because of that.

\subsection{Preliminary estimates}\label{subsec:preliminary-estimates}

The regularity requirement on $a$ imposes exponential decay on its Fourier coefficients. This will turn out to be crucial for the estimation of $||g||_{\varrho - \delta}$.
\begin{lemma}[Exponential decay of Fourier terms in $\mathcal{P}_0(\varrho)$, see e.g. \cite{arnold}]\label{lem:exponential-decay-of-fourier-terms}
 Let $b \in \mathcal{P}_0(\varrho)$ be a function with Fourier coefficients $b_{p,q}, (p,q) \in \mathbb{Z}^2_*$. Then
 \begin{equation}\label{eq:exponential-decay-of-fourier-coefficients-1-param}
   |b_{p,q}| \leqslant ||b||_\varrho \cdot e^{-(|p|+|q|)\varrho}.
 \end{equation}
\end{lemma}
We are now ready to present the
\begin{proof}[Proof of the second part of lemma \ref{lem:cohomological-estimates-norm-times-series}]
 First observe that for $\varrho > 0$ and $p, q \in \mathbb{Z}$ the norm of the Fourier basis element is given by $||e_{p,q}|| = e^{(|p| + |q|)\varrho}.$ To see this observe that the function $(x,y) \mapsto \left| e^{i(px-qy)} \right|$ is really just a function of $\Im x$ and $\Im y$ and is monotone in both of these variables if we leave the other one fixed, with the monotonicity being dependent solely on the signs of $p$ and $q$. Therefore it attains its maximum either on the top or bottom segments of $\overline{\Pi(\varrho)}$, that is for $x = \Re x \pm i\varrho$ and $y = \Re y \pm i\varrho$. Plugging these values into the function and choosing the sign correctly yields the maximal value of $e_{p,q}$ stated above.

 Considering this and formula \eqref{eq:gpq} for $g_{p,q}$ the estimation of $||g||_{\varrho-\delta}$ reads
  \begin{align}
  \begin{split}
    \sup_{x,y \in \overline{\Pi(\varrho - \delta)}} |g(x,y)| &\leqslant \sum_{(p,q) \in \mathbb{Z}^2_*} {|a_{p,q}| \over |q\omega - p|} ||e_{p,q}||_{\varrho - \delta} = \sum_{(p,q) \in \mathbb{Z}^2_*} {|a_{p,q}| \over |q\omega - p|} e^{(\varrho - \delta)(|p| + |q|)} \stackrel{(\star)}{\leqslant} \\
    &\stackrel{(\star)}{\leqslant} ||a||_\varrho \cdot \sum_{(p,q)\in\mathbb{Z}^2_*} {e^{-(|p|+|q|)\varrho} \cdot e^{(\varrho - \delta)(|p| + |q|)} \over |q \omega - p|} = ||a||_\varrho \cdot \ser(\delta).
  \end{split}
  \end{align}
 Inequality $(\star)$ is a direct consequence of lemma \ref{lem:exponential-decay-of-fourier-terms}.
\end{proof}

\subsection{Estimates in the general case}

We now focus on proving theorem \ref{thm:general-estimate}. As we will often use the ``subseries'' of $\ser(\delta)$ over smaller sets of indices, we will therefore denote them as
\begin{equation}
 \Sfrak(A) := \sum_{(q,p) \in A} L(q,p)
\end{equation}
for any $A \subset \mathbb{Z}_*^2$.

The subsets of $\mathbb{Z}_*^2$ that will be the most essential for us are the following:
\begin{itemize}
 \item the strips parallel to the $p = \omega q$ line:
 \begin{equation}
  \Strip(n) := \{ (q,p) \in \mathbb{Z}_*^2 : n < q\omega - p < n+1 \}
 \end{equation}
 with $n \in \mathbb{Z}$,
 \item the critical strip around the $p = \omega q$ line:
 \begin{equation}
  \CritStrip := \{ (q,p) \in \mathbb{Z}_*^2 : |q \omega - p| < 1 \} = \Strip(-1) \cup \Strip(0),
 \end{equation}
 \item its complement in $\mathbb{Z}_*^2$:
 \begin{equation}
  \Away := \{ (q,p) \in \mathbb{Z}_*^2 : |q \omega - p| > 1 \} = \bigcup_{n \in \mathbb{Z} \setminus \{ -1, 0 \}} \Strip(n),
 \end{equation}
 \item the Brjuno indices\footnote{note that we use $\Bad$ written in sans-serif font for Brjuno \emph{indices} and $\Brjuno$ written in typewriter font for Brjuno \emph{numbers}.} $\Bad = \Bad^+ \cup -\Bad^+$ contained within $\CritStrip$ for which we will have no control without an additional arithmetic assumption on $\omega$, where $\Bad^+$ is given by
 \begin{equation}\label{eq:brjuno-indices-definition}
  \Bad^+ := \left\{ (q,p) \in \mathbb{Z}^2_* : q = aq_k, p = \nint(q\omega) \mbox{ for some } k \in \mathbb{N} \mbox{ and } 1 \leqslant a \leqslant a^*_{k+1} \right\},
 \end{equation}
 \item the remaining ``constant type'' indices $\Good$ for which $|q\omega - p| \geqslant (2|q|)^{-1}$ by Legendre's theorem \ref{thm:legendre}:
 \begin{equation}
  \Good := \CritStrip \setminus \Bad.
 \end{equation}
\end{itemize}

The estimates will be performed separately on $\Away, \Good$ and $\Bad$ in the following three subsections of this section. These sets form a partition of $\mathbb{Z}^2_*$, therefore $\ser(\delta)$ will be equal to the sum of the three series combined, each of which will be further estimated by the three summands in the definition of $\Gamma(\delta)$ in \eqref{eq:Gamma-of-delta}.

In each case the upper bound on $\Sfrak( \ldots )$ will be of the form $G \cdot f(\delta)$ with a constant $G$ and $f(\delta) = \delta^{-u} \cdot P(\log (\delta^{-1}))$ for some exponent $u > 0$ and some monic polynomial $P$ of degree at most $2$. We emphasize, however, that arriving at a specific $G$ is both very tedious and more impotantly requires a choice of an upper cut-off threshold on $\delta$ for the inequality to be valid. In practice we are only interested in $\delta > 0$ small, therefore we will write $G$ as an absolute constant plus an increment that tends to $0$ as $\delta \to 0$.

\subsubsection{Indices away from the critical line}\label{subsubsec:away}

\begin{lemma}\label{lem:estimates-on-away}
 The series $\Sfrak(\Away)$ satisfies the estimate
 \begin{equation}
  \Sfrak(\Away) < G_\Away \cdot \delta^{-1} \log \left( \delta^{-1} \right)
 \end{equation}
 with
 \begin{equation}\label{eq:GAway}
  G_\Away = {4 \over 1 + \omega} + {2 \over 1 - \omega} + O(\delta)
 \end{equation}
 for small enough $\delta > 0$.
\end{lemma}

\begin{proof}
We begin by observing that if $n$ and $q$ are fixed there is exactly one $p$ such that $(q,p) \in \Strip(n)$. This $p$ is given by $p = \lfloor q\omega \rfloor - n$ since the inequalities in the defition of $\Strip(n)$, namely $n < q\omega - p < n+1$, are equivalent to $q\omega - n > p > q\omega - (n+1)$ and this is a constraint on $p$ to an interval of length $1$.

This way $\Sfrak(\Strip(n))$ becomes
\begin{align}\label{eq:SfrakStrip-n}
\begin{split}
 \Sfrak(\Strip(n)) &= \sum_{(q,p) \in \Strip(n)} L(q,p) = \sum_{q \in \mathbb{Z}} L(q, \lfloor q\omega \rfloor - n) = \sum_{q \in \mathbb{Z}} { e^{-(|q| + |\lfloor q\omega \rfloor - n|) \delta} \over |q\omega - p| } = \sum_{q \in \mathbb{Z}} { e^{-(|q| + |\lfloor q\omega \rfloor - n|) \delta} \over |q\omega - (\lfloor q\omega \rfloor - n)| }.
\end{split}
\end{align}

We should separately consider cases of $n \leqslant -2$ and $n \geqslant 1$, we will, however, do so only for the second one since the sum over all sets $\Strip(n)$ with $n \leqslant -2$ is exactly the same as the sum over all sets $\Strip(n)$ with $n \geqslant 1$ - it is only a matter of changing the summation indices. Indeed, consider $n \geqslant 1$ and $q \in \mathbb{Z}$ and define an index change by $(n,q) = (-n'-1, -q')$. The image of this transformation is the set $\{ n' \leqslant -2, q' \in \mathbb{Z} \}$ and the summand is left intact:
\begin{align}
\begin{split}
 { e^{-(|q| + |\lfloor q\omega \rfloor - n|) \delta} \over |q\omega - (\lfloor q\omega \rfloor - n)| } &= { e^{-(|-q'| + |\lfloor -q'\omega \rfloor - (-n'-1)|) \delta} \over |-q\omega - (\lfloor -q\omega \rfloor - (-n'-1))| } = { e^{-(|q'| + |-\lceil q'\omega \rceil + n' + 1)|) \delta} \over |-q'\omega + \lceil q'\omega \rceil -n'-1| } \\
 &= { e^{-(|q'| + |\lceil q'\omega \rceil - n' - 1)|) \delta} \over |q'\omega - \lceil q'\omega \rceil + n' + 1| } = { e^{-(|q'| + |\lfloor q'\omega \rfloor + 1 - n' - 1)|) \delta} \over |q'\omega - \lfloor q'\omega \rfloor - 1 + n' + 1| } = { e^{-(|q'| + |\lfloor q'\omega \rfloor - n')|) \delta} \over |q'\omega - (\lfloor q'\omega \rfloor - n')| }.
\end{split}
\end{align}

This shows that the sum over all $\Away$ is actually the doubled sum over all strips $\Strip(n)$ with $n \geqslant 1$. We will thus only focus on the $n \geqslant 1$ case.

\newcommand{\Qsf}{\mathsf{Q}}
Now, to rid ourselves of the absolute values in the exponent of the last expression in \eqref{eq:SfrakStrip-n} we have to split $\Sfrak(\Strip(n))$ into sums over smaller sets. Denote, for a fixed $n \in \mathbb{Z} \setminus \{ -1, 0 \}$, the sets
\begin{align}
\begin{split}
 \Qsf_{++} (n) &= \{ q : q \geqslant 0 \wedge \lfloor q\omega \rfloor - n \geqslant 0 \}, \qquad \Qsf_{+-} (n) = \{ q : q \geqslant 0 \wedge \lfloor q\omega \rfloor - n < 0 \}, \\
 \Qsf_{-+} (n) &= \{ q : q < 0 \wedge \lfloor q\omega \rfloor - n \geqslant 0 \}, \qquad \Qsf_{--} (n) = \{ q : q < 0 \wedge \lfloor q\omega \rfloor - n < 0 \}.
\end{split}
\end{align}

First observe that if $n \geqslant 1$ the set $\Qsf_{-+} (n)$ is empty and the remaining ones are traces of intervals on $\mathbb{Z}$ given by $\Qsf_{++} (n) = [ \lfloor n / \omega \rfloor , \infty ) \cap \mathbb{Z}, \Qsf_{+-} (n) = [ 0, \lfloor n / \omega \rfloor - 1 ] \cap \mathbb{Z}$ and $\Qsf_{--} (n) = (- \infty, -1] \cap \mathbb{Z}$.

We will consider the last series in \eqref{eq:SfrakStrip-n} separately on each of these three sets. Recall that by definition of $\Strip(n)$ we have $|q\omega - p| > n$ for each $(q,p) \in \Strip(n)$. The sum over $\Qsf_{++}(n)$ becomes
\begin{align}
\begin{split}
 \sum_{q \in \Qsf_{++}(n)} { e^{-(|q| + |\lfloor q\omega \rfloor - n|) \delta} \over |q\omega - p| } &= \sum_{q = \lfloor n/w \rfloor}^\infty { e^{-(q + \lfloor q\omega \rfloor - n) \delta} \over |q\omega - p| } < {e^{n\delta} \over n} \sum_{q = \lfloor n/\omega \rfloor}^\infty e^{-(q + \lfloor q\omega \rfloor) \delta} < {e^{n \delta} \over n} \sum_{q = \lfloor n/w \rfloor}^\infty e^{-(q + q\omega - 1) \delta} = \\ &= {e^{(n+1)\delta} \over n} \sum_{q = \lfloor n/\omega \rfloor}^\infty e^{-q(1 + \omega) \delta} = {e^{(n+1)\delta} \over n} \cdot { e^{-(1+\omega) \left\lfloor {n \over \omega} \right\rfloor \delta} \over 1 - e^{-(1+\omega)\delta }} \leqslant {e^{(n+1)\delta} \over n} \cdot { e^{-(1+\omega) \left( {n \over \omega} - 1 \right) \delta} \over 1 - e^{-(1+\omega)\delta }} = \\
 &= { e^{(2+\omega)\delta} \cdot e^{-(n/\omega) \delta} \over n \cdot (1 - e^{-(1+\omega)\delta}) }.
\end{split}
\end{align}

Now observe that for small enough $\delta > 0$ we have $e^{(2+\omega)\delta} \approx 1 + (2+\omega) \delta$ and $1 - e^{-(1+\omega)\delta} \approx (1+\omega) \delta$. This way with a simple argument one can - again, for small enough $\delta$ - show that
\begin{equation}\label{eq:Q++-n-positive-final}
 \sum_{q \in \Qsf_{++}(n)} { e^{-(|q| + |\lfloor q\omega \rfloor - n|) \delta} \over |q\omega - p| } < {1 + O(\delta) \over 1 + \omega} \cdot \delta^{-1} \cdot {\left( e^{-\delta/\omega} \right)^n \over n}.
\end{equation}

Recall also that the Maclaurin expansion for $-\log(1-x)$ is $-\log(1-x) = x + x^2/2 + x^3/3 + x^4/4 + \ldots$, which combined with \eqref{eq:Q++-n-positive-final} gives the following estimate:
\begin{equation}
 \sum_{n=1}^\infty \sum_{q \in \Qsf_{++}(n)} { e^{-(|q| + |p|) \delta} \over |q\omega - p| } \leqslant \left( {1 \over 1 + \omega} + O(\delta) \right) \cdot \delta^{-1} \cdot \left(  -\log \left( 1 - e^{-\delta/\omega} \right) \right)
\end{equation}
(in the above $p = \lfloor q\omega \rfloor - n$ as introduced earlier).

For $\delta > 0$ small enough we have $1 - \delta \geqslant e^{-\delta/\omega}$ if $0 < \omega < 1$ and this gives $-\log \left( 1 - e^{-\delta/\omega} \right) \leqslant \log \left( \delta^{-1} \right)$. Eventually because of that the estimates read
\begin{equation}
 \sum_{n=1}^\infty \sum_{q \in \Qsf_{++}(n)} { e^{-(|q| + |\lfloor q\omega \rfloor - n|) \delta} \over |q\omega - p| } \leqslant \left( {1 \over 1 + \omega} + O(\delta) \right) \cdot \delta^{-1} \log \left( \delta^{-1} \right).
\end{equation}

We can similarly deal with the sum over $\Qsf_{--}(n)$:
\begin{align}
\begin{split}
 \sum_{q \in \Qsf_{--}(n)} { e^{-(|q| + |\lfloor q\omega \rfloor - n|) \delta} \over |q\omega - p| } &\leqslant \sum_{q = -\infty}^{-1} { e^{-((-q) + (-\lfloor q\omega \rfloor + n)) \delta} \over |q\omega - p| } < {1 \over n} \sum_{q = -\infty}^{-1} e^{-((-q) + (-\lfloor q\omega \rfloor + n)) \delta} \stackrel{(q \to -q)}{=} \\
 &= {1 \over n} \sum_{q = 1}^\infty e^{-(q + \lceil q\omega \rceil + n)\delta} < {e^{-n\delta} \over n} \sum_{q=1}^\infty e^{-q(1+\omega)\delta} = {e^{-n \delta} \over n} \cdot {e^{-(1+\omega)\delta} \over 1 - e^{-(1+\omega)\delta}}.
\end{split}
\end{align}

Using an analogous reasoning we can infer that for small enough $\delta$ we have
\begin{equation}
  \sum_{q \in \Qsf_{--}(n)} { e^{-(|q| + |\lfloor q\omega \rfloor - n|) \delta} \over |q\omega - p| } \leqslant \left( {1 \over 1 + \omega} + O(\delta) \right) \cdot \delta^{-1} \cdot {\left(e^{-\delta}\right)^n \over n}
\end{equation}
and eventually
\begin{align}
\begin{split}
  \sum_{n=1}^\infty \sum_{q \in \Qsf_{--}(n)} { e^{-(|q| + |\lfloor q\omega \rfloor - n|) \delta} \over |q\omega - p| } &\leqslant \left( {1 \over 1 + \omega} + O(\delta) \right) \cdot \delta^{-1} \cdot \left(-\log \left( 1 - e^{-\delta} \right) \right) = \\ 
  &=\left( {1 \over 1 + \omega} + O(\delta) \right) \cdot \delta^{-1} \log \left( \delta^{-1} \right),
\end{split}
\end{align}
where the last equality is valid for small enough $\delta$.\footnote{note that the $O(\delta)$ terms on the two sides of this equality are not the same}

We now consider the sum over $\Qsf_{+-}(n)$.
\begin{align}
\begin{split}
 \sum_{q \in \Qsf_{+-}(n)} { e^{-(|q| + |\lfloor q\omega \rfloor - n|) \delta} \over |q\omega - p| } &< {1 \over n} \sum_{q=0}^{\lfloor n/\omega \rfloor - 1} e^{-(q - \lfloor q\omega \rfloor + n)\delta} < {e^{-n\delta} \over n} \sum_{q=0}^{\lfloor n/\omega \rfloor - 1} e^{-q(1-\omega)\delta} = \\
 &= {e^{-n\delta} \over n} \cdot {1 - e^{-(1-\omega)\lfloor n/\omega \rfloor \delta} \over 1 - e^{-(1-\omega)\delta} } < {e^{-n\delta} \over n} \cdot {1 \over 1 - e^{-(1-\omega)\delta} }.
\end{split}
\end{align}

We can now - analogously to the $\Qsf_{++}(n)$ and $\Qsf_{--}(n)$ cases - infer that
\begin{equation}
 \sum_{n = 1}^\infty \sum_{q \in \Qsf_{+-}(n)} { e^{-(|q| + |\lfloor q\omega \rfloor - n|) \delta} \over |q\omega - p| } < \left( {1 \over 1-\omega} + O(\delta) \right) \delta^{-1} \log \left( \delta^{-1} \right).
\end{equation}

Adding the three sums together gives the estimate $\Sfrak \left( \bigcup_{n \geqslant 1} \Strip(n) \right) < G_+ \cdot \delta^{-1} \log \left( \delta^{-1} \right)$ with $G_+ = {2 \over 1 + \omega} + {1 \over 1 - \omega} + O(\delta)$. The desired inequality now follows from the fact we stated earlier in the proof, namely that $\Sfrak(\Away) = 2\cdot \Sfrak \left( \bigcup_{n \geqslant 1} \Strip(n) \right)$.
\end{proof}

\subsubsection{Indices of ``constant type''}\label{subsubsec:constanttype}

\begin{lemma}\label{lem:good-series-estimate}
 The series $\Sfrak(\Good)$ satisfies the estimate
 \begin{equation}
  \Sfrak(\Good) \leqslant G_\Good \cdot \delta^{-2}
 \end{equation}
 for small enough $\delta > 0$. Here $G_\Good$ is given by
 \begin{equation}\label{eq:GGood}
  G_\Good = {8 \over (1+\omega)^2} + O(\delta).
 \end{equation}
\end{lemma}

\begin{proof}
 First observe that
 \begin{equation}
  \Sfrak(\Good) = \sum_{(q,p) \in \Good} {e^{-(|q|+|p|)\delta} \over |q\omega - p|} \leqslant  \sum_{(q,p) \in \Good} 2|q|e^{-(|q|+|p|)\delta}.
 \end{equation}
 because of Legendre's theorem \ref{thm:legendre}. Since $\Good \subset \CritStrip$ we can estimate further by
 \begin{equation}
  \Sfrak(\Good) \leqslant \sum_{(q,p) \in \CritStrip} 2|q|e^{-(|q|+|p|)\delta}.
 \end{equation}
 The set $\CritStrip$ is symmetric with respect to $(0,0)$ and the summand in the last sum does not change under $(q,p) \mapsto (-q,-p)$, therefore we can write
 \begin{align}
 \begin{split}
  \Sfrak(\Good) &\leqslant 2 \cdot \sum_{(q,p) \in \CritStrip, q \geqslant 1} 2|q|e^{-(|q|+|p|)\delta} = 4 \cdot \sum_{(q,p) \in \CritStrip, q \geqslant 1} qe^{-(q+p)\delta} = \\
  &= 4 \cdot \sum_{q=1}^\infty q \left( e^{-(q+\lfloor q\omega \rfloor)\delta} + e^{-(q+\lceil q\omega \rceil)\delta} \right) < 4 \cdot \sum_{q=1}^\infty q \left( e^{-(q+ q\omega -1)\delta} + e^{-(q+ q\omega )\delta} \right) = \\
  &= 4(1 + e^\delta) \sum_{q=1}^\infty qe^{-q(1+\omega)\delta} = 4(1+e^\delta) {e^{-(1+\omega)\delta} \over (1 - e^{-(1+\omega)\delta})^2} = \left({8 \over (1+\omega)^2} + O(\delta) \right) \delta^{-2}.
 \end{split}
 \end{align} 
\end{proof}

\subsubsection{Brjuno indices}\label{subsubsec:brjuno}

Since Brjuno indices are of the form $(aq_k, \nint (aq_k\omega))$ for $k \in \mathbb{N}$ and $1 \leqslant a \leqslant a_{k+1}^*$ we will first write $\nint(aq_k\omega)$ in terms of $p_k$, the numerator of the $k$-th convergent to $\omega$.

\begin{lemma}
 If $k \in \mathbb{N}$ and $1 \leqslant a \leqslant a^*_{k+1}$ then $\nint(aq_k\omega) = ap_k$.
\end{lemma}

\begin{proof}
 To prove the lemma we have to verify whether $|aq_k\omega - ap_k| < 1/2$ for the required $k$ and $a$. To do this we will use two well known facts from the theory of continued fractions: the estimate $|q_k \omega - p_k| < q_{k+1}^{-1}$, valid for all $k \geqslant 0$ and the recursive formulas for $q_k$, namely $q_k = a_k q_{k-1} + q_{k-2}, k \geqslant 0$ with $q_{-2} = 1$ and $q_{-1} = 0$. We have
 \begin{equation}\label{eq:nint-estimate}
  |aq_k\omega - ap_k| \stackrel{(\star)}{<} {a \over q_{k+1}} = {a \over a_{k+1} q_k + q_{k-1}} < {a \over a_{k+1}} \cdot {1 \over q_k} < {\sqrt{(a_{k+1}+2)/2} \over a_{k+1}} \cdot {1 \over q_k} \leqslant {\sqrt{3/2} \over q_k}.
 \end{equation}
 From the recursive formulas we see that $q_1 = a_1$, $q_2 = a_2a_1 + 1$ and that $(q_j)_{j=0}^\infty$ is an increasing sequence. This way whenever $k \geqslant 3$ we can further estimate by $1/2$ in \eqref{eq:nint-estimate} as $q_3 \geqslant 3$. For $k \in \{ 0, 1, 2 \}$ we need to consider separate cases and investigate the initial estimate in \eqref{eq:nint-estimate}, marked with $(\star)$.
 
 Observe that whenever $a_{k+1} = 1$ then $a^*_{k+1} = 1$ and therefore $a = 1$. This way inequality $(\star)$ is valid whenever $k \in \{ 1, 2 \}$ since $q_{2+1} > q_{1+1} \geqslant 2$. For $k = 0$ we have $a_{k+1} = a_1 = 1$ and this way this case concerns only $\omega < 1/2$. The estimates in \eqref{eq:nint-estimate} therefore read $|aq_k \omega - ap_k| = |q_0 \omega - p_0| = \omega < 1/2$.
 
 If $a_{k+1} \geqslant 2$ we only need to consider the cases of $k = 0$ and $k = 1$ separately, since for $k \geqslant 2$ we have $q_k \geqslant 2 > \sqrt{2}$ and this way, similarly to \eqref{eq:nint-estimate} we can estimate $|aq_k\omega - ap_k| < \sqrt{(a_{k+1} + 2)/2}/(a_{k+1}q_k) < 1/2$ since the minimal value of $\sqrt{(a_{k+1} + 2)/2}/a_{k+1}$ is $\sqrt{2}/2$. 
 
 If $k = 1$ and $a_{k+1} = a_2 \geqslant 2$ we have
 \begin{equation}
  |aq_1\omega - ap_1| < {a \over q_2} = {a \over a_2a_1 + 1} \leqslant {a \over a_2 + 1} < \sqrt{a_2 + 2 \over 2(a_2 + 1)^2} \leqslant {\sqrt{2} \over 3} < {1 \over 2}.
 \end{equation}
 If $k = 0$ and $a_{k+1} = a_1 \geqslant 2$ then similarly
 \begin{equation}
  |aq_0 \omega - ap_0| < {a \over q_1} = {a \over a_1}.
 \end{equation}
 If $a_1 = 2$ then $a^*_1 = \lfloor \sqrt{(2+2)/2} \rfloor = 1$ and therefore $a = 1$ and $a/a_1 = 1/2$, similarly for $a_1 = 3$ we have $a/a_1 = 1/3$. For $a_1 \geqslant 4$ we can further estimate by
 \begin{equation}
  {a \over a_1} \leqslant \sqrt{a_1 + 2 \over 2a_1^2} \leqslant \sqrt{6 \over 32} < {1 \over 2}.
 \end{equation}
\end{proof}

Knowing the form of $\nint(aq_k\omega)$ we can now proceed to estimating $\Sfrak(\Bad)$. We will do this in terms of the Brjuno-like functions given in \eqref{eq:brjuno-functions-definition}.

The two Brjuno-like functions correspond to two separate types of Brjuno indices introduced in \eqref{eq:brjuno-indices-definition} - the first one to $a=1$ and the second one to $2 \leqslant a \leqslant a^*_{k+1}$. We elaborate on this vague statement in the remaining part of this subsection.

Define
\begin{equation}
 \Bad^+_1 := \left\{ (q,p) \in \mathbb{Z}^2_* : q = q_k, p = p_k \mbox{ for some } k \in \mathbb{N} \right\}
\end{equation}
and
\begin{equation}
 \Bad^+_{2+} := \left\{ (q,p) \in \mathbb{Z}^2_* : q = aq_k, p = ap_k \mbox{ for some } k \in \mathbb{N} \mbox{ and } 2 \leqslant a \leqslant a^*_{k+1} \right\}.
\end{equation}

\begin{lemma}
 The following estimate is valid for $\Sfrak(\Bad)$:
 \begin{equation}\label{eq:brjuno-series-estimate}
  \Sfrak(\Bad) \leqslant 2 \cdot [(2+O(\delta)) \Brj_1(\Delta) + (1+O(\delta))\Brj_2(2\Delta)],
 \end{equation}
 where $\Delta = (1+\omega)\delta$.
\end{lemma}

\begin{proof}
 First observe that the set $\Bad$ is symmetric with respect to $(0,0)$ as $\Bad = \Bad^+ \cup -\Bad^+$ and that $L(q,p) = L(-q,-p)$ for all nonzero pairs $(p,q)$. This way $\Sfrak(\Bad) = 2\Sfrak(\Bad^+)$ which gives the constant $2$ in \eqref{eq:brjuno-series-estimate} before the square brackets. What remains to be proven is the estimate of $\Sfrak(\Bad^+)$ by the contents in the square brackets. To do this we will - as indicated before the formulation of the lemma - use the two ingredients separately and prove that $\Sfrak(\Bad^+_1) \leqslant (2+O(\delta))\Brj_1(\Delta)$ and $\Sfrak(\Bad^+_{2+}) \leqslant (1+O(\delta))\Brj_2(2\Delta)$.
 
 {\bf The sum over $\Bad^+_1$.} All pairs $(p,q) \in \Bad^+$ satisfy $|q\omega - p| < 1$ and in particular $p > q\omega - 1$. This way for $(p,q) = (p_k, q_k) \in \Bad^+_1$ we can estimate
 \begin{align}
 \begin{split}
  L(q,p) &= {e^{-(q+p)\delta} \over |q\omega - p|} < {e^{-(q+q\omega-1)\delta} \over |q\omega - p|} = e^\delta {e^{-q_k\Delta} \over |q_k \omega - p_k|} < e^\delta \cdot 2q_{k+1} e^{-q_k\Delta} = (2 + O(\delta)) e^{-q_k\Delta} q_{k+1}.
 \end{split}
 \end{align}
 Summing over $k \geqslant 1$ concludes this case.
 
 {\bf The sum over $\Bad^+_{2+}$.} For $(p,q) = (ap_k, aq_k) \in \Bad^+_{2+}$ with $2 \leqslant a \leqslant a^*_{k+1}$ we can also write $p > q\omega - 1$. The estimates of $L(q,p)$ for such $(p,q)$ read
 \begin{align}
 \begin{split}
  L(q,p) &= {e^{-(q+p)\delta} \over |q\omega - p|} < {e^{-(q+q\omega-1)\delta} \over |q\omega - p|} = e^\delta {e^{-aq_k\Delta} \over a|q_k\omega - p_k|} \leqslant e^\delta {e^{-2q_k\Delta} \over a|q_k\omega - p_k|} < 2e^\delta {1 \over a} e^{-2q_k\Delta} q_{k+1}.
 \end{split}
 \end{align}
 Summing over all $(p,q)$ is now equivalent to summing over all $2 \leqslant a \leqslant a^*_{k+1}$ first and then over all $k \geqslant 1$. In the above estimate there is only one term that depends explicitly on $a$, namely $a^{-1}$. The sum of $a^{-1}$ is the $a^*_{k+1}$-st harmonic number decremented by 1, which can be further estimated from above by $\log a^*_{k+1}$. A straightforward verification gives the estimate $a^*_{k+1} \leqslant \sqrt{a_{k+1}}$ which, after summing over all $k \geqslant 1$, gives the upper bound of the whole $\Sfrak(\Bad^+_{2+})$ by the desired $(1+O(\delta))\Brj_2(2\Delta)$.
\end{proof}

\subsection{Estimates for Diophantine numbers}\label{subsec:dioph}

To estimate $\Brj_1(\Delta)$ and $\Brj_2(\Delta)$ we use the following recursive upper bound on $q_n$ (for a proof of a slight variant of this result see \cite{kamienski:2018-khintchine-levy}):

\begin{lemma}[Diophanticity in terms of the continued fraction expansion]\label{lem:diophanticity-in-terms-of-continued-fraction}
 If an irrational number $\omega$ is $(C,\tau)$-Diophantine with $C>0$ and $\tau \geqslant 1$ then the denominators of its convergents and its partial quotients can be estimated by
 \begin{equation}\label{eq:diophanticity-in-terms-of-continued-fraction}
  q_{n+1} \leqslant C^{-1} q_n^{\tau} \qquad \mbox{ and } \qquad a_{n+1} \leqslant C^{-1} q_n^{\tau - 1}.
 \end{equation}
 Conversely, estimates as in \eqref{eq:diophanticity-in-terms-of-continued-fraction} for all $n \geqslant 0$ result in $\omega$ being $(C/(2+C), \tau)$-Diophantine.
\end{lemma}

\newcommand{\Dph}{\mathfrak{Dph}}
Using this result we can readily infer that $\Brj_1(\Delta) \leqslant C^{-1} \Dph_1(\Delta)$ and $\Brj_2(\Delta) \leqslant \log(C^{-1}) \cdot \Dph_1(\Delta) + C^{-1}(\tau - 1) \Dph_2(\Delta)$, where $\Dph_1$ and $\Dph_2$ ar given by
\begin{equation}
 \Dph_1(\Delta) := \sum_{n=1}^\infty e^{-q_n\Delta} q_n^\tau \qquad \mbox{ and } \qquad \Dph_2(\Delta) := \sum_{n=1}^\infty e^{-q_n\Delta} q_n^\tau \log q_n.
\end{equation}

The proof of theorem \ref{thm:brj12-for-diophantine} reduces therefore to suitably estimating both $\Dph_1(\Delta)$ and $\Dph_2(\Delta)$. We will do this in the lemmas that follow.

In the proofs in the remaining part of this subsection we will often estimate series by appropriate improper integrals. This is possible whenever a function $f : [R - 1, \infty) \mapsto \mathbb [0, \infty)$ with $R \in \mathbb{N}$ is nonincreasing - then we have $\sum_{n=R}^\infty f(n) \leqslant \int_{R-1}^\infty f(x) \, \integrald x$.

\begin{lemma}\label{lem:dph1-estimates}
 The series $\Dph_1(\Delta)$ satisfies
 \begin{equation}\label{eq:dph1-estimate}
  \Dph_1(\Delta) \leqslant {(\tau e^{-1})^\tau \over \log \varphi} \cdot \Delta^{-\tau} \left( \log (\Delta^{-1}) + G_\mathtt{Dph1}^{(0)} \right),
 \end{equation}
 where $G_\mathtt{Dph1}^{(0)} = \log(3\tau\varphi) + \GammaEul(\tau) / (2 (\tau e^{-1})^\tau)$.
\end{lemma}

\begin{proof}
 Denote $\tilde F_k := \varphi^k/3$. We have $\tilde F_k < q_k$ as both $\tilde F_k < F_k$ and $F_k < q_k$ hold for all $k$, where $(F_k)$ is the Fibonacci sequence. The function $\mathfrak{D}_1(t) := t^\tau e^{-t\Delta}$ has a single maximum at $t_1 := \tau \Delta^{-1}$ with $\mathfrak{D}_1(t_1) = (\tau e^{-1})^\tau \Delta^{-\tau}$. This way for $k$ for which $t_1 \leqslant \tilde F_k < q_k$, that is $k \geqslant k_* := \lceil \log_\varphi (3\tau \Delta^{-1}) \rceil$, we can estimate $\mathfrak{D}_1(q_k) < \mathfrak{D}_1(\tilde F_k)$ and thus
 \begin{equation}
  \sum_{k=k_*}^\infty \mathfrak{D}_1(q_k) < \sum_{k=k_*}^\infty \mathfrak{D}_1(\tilde F_k) = 3^{-\tau} \sum_{k=k_*}^\infty \varphi^{\tau k} e^{-\varphi^k \Delta / 3}
 \end{equation}
 as $\mathfrak{D}_1$ is decreasing for these $k$. This is also the case with the function $f_1(x) := 3^{-\tau} \varphi^{\tau x} e^{-\varphi^x \Delta / 3}$, therefore the last sum can be further estimated by an improper integral of $f_1$ over $[k_* - 1, \infty)$. To estimate the integral it is best to first split it into a sum of two - one over $[k_* - 1, \log_\varphi (3\tau \Delta^{-1}))$ and another over $[\log_\varphi (3\tau \Delta^{-1}), \infty)$.
 
 The first of these integrals can be estimated by the maximal value of $f_1$, which is $f_1(\log_\varphi(3\tau\Delta^{-1})) = (\tau e^{-1})^\tau \Delta^{-\tau}$. This is because the first domain of integration is of length at most $1$.
 
 The second, improper integral can be transformed by means of a substitution $y = \varphi^x \Delta/3$ into
 \begin{equation}
  3^{-\tau} \int_{\log_\varphi(3\tau\Delta^{-1})}^\infty \varphi^{\tau x} e^{-\varphi^x \Delta/3} \, \integrald x = \int_\tau^\infty {\Delta^{-\tau} \over \log \varphi} y^{\tau - 1} e^{-y} \, \integrald y \leqslant {1 \over 2} {\GammaEul(\tau) \over \log \varphi} \Delta^{-\tau}.
 \end{equation}
 The $1/2$ factor in the last inequality stems from the fact that the median $\mathsf{mdn}(\tau)$ of the Gamma distribution with shape parameter $\alpha = \tau$ and rate parameter $\beta = 1$ (given by density $\mathsf{dns}(x; \alpha, \beta) = \beta^\alpha x^{\alpha - 1} e^{-\beta x} / \GammaEul(\alpha)$) satisfies $\tau - 1/3 < \mathsf{mdn}(\tau) < \tau$ (see \cite{chen-rubin}).

 The initial part of the sum, i.e. $\sum_{k=1}^{k_* - 1} \mathfrak{D}_1(q_k)$, can be naively estimated by the upper bound over the values of the summands multiplied by the upper bound on their number:
 \begin{equation}
  \sum_{k=1}^{k_* - 1} \mathfrak{D}_1(q_k) \leqslant (\tau e^{-1})^\tau \Delta^{-\tau} \cdot \log_\varphi (3\tau\Delta^{-1}).
 \end{equation}
 Adding the three ingredients (short integral, improper integral and initial part of the sum) together gives the desired estimate \eqref{eq:dph1-estimate}.
\end{proof}

\begin{lemma}\label{lem:dph2-estimates}
 If $\Delta \leqslant \min \{ \tau^{-1}, \tau e^{-1} \}$ then the series $\Dph_2(\Delta)$ satisfies
 \begin{equation}
  \Dph_2(\Delta) \leqslant {(\tau e^{-1})^\tau \over \log \varphi} \cdot \Delta^{-\tau} \left( \log^2 (\Delta^{-1}) + G_\mathtt{Dph2}^{(1)} \log (\Delta^{-1}) + G_\mathtt{Dph2}^{(0)}\right),
 \end{equation}
 where $G_\mathtt{Dph2}^{(1)} = \log (3\varphi(\tau+1)^2) + \GammaEul(\tau)/(2(\tau e^{-1})^\tau)$ and $G_\mathtt{Dph2}^{(0)} = \log(3\varphi(\tau+1))\log(\tau+1) + \GammaEul'(\tau)/(\tau e^{-1})^\tau$.
\end{lemma}

\begin{proof}
 The strategy will be similar to the one undertaken in the proof of lemma \ref{lem:dph1-estimates}, only slightly more intricate due to the fact that instead of the function $\mathfrak{D}_1(t)$ there appears $\mathfrak{D}_2(t) := e^{-t\Delta} t^\tau \log t$.
 
 Again, denote $\tilde F_k = \varphi^k/3$ and observe that $\mathfrak{D}_2$ has a single maximum $t_2$. This time, however, we cannot express it through an explicit formula like in the proof of lemma \ref{lem:dph1-estimates} due to the $\log t$ term in $\mathfrak{D}_2$, but we can still determine an approximate location of $t_2$. We have $t_2 \in (t_\mathtt{lo}, t_\mathtt{hi})$, where $t_\mathtt{lo} = t_1 = \tau \Delta^{-1}$ and $t_\mathtt{hi} = t_\mathtt{lo} + \Delta^{-1}/\log(\tau \Delta^{-1})$. Indeed, observe that $\mathfrak{D}_2'(t_\mathtt{lo}) > 0$ and $\mathfrak{D}_2'(t_\mathtt{hi}) < 0$ with the sign changing exactly once in the interval.
 
 Define $\bar t_\mathtt{hi} := (\tau + 1)\Delta^{-1}$ and observe that when $\Delta$ satisfies the smallness assumption from the formulation of the lemma we have $t_\mathtt{hi} \leqslant \bar t_\mathtt{hi}$. As a consequence on $[\bar t_\mathtt{hi}, \infty)$ the function $\mathfrak{D}_2$ is decreasing, in particular whenever $\bar t_\mathtt{hi} \leqslant \tilde F_k < q_k$ we have $\mathfrak{D}_2(q_k) < \mathfrak{D}_2(\tilde F_k)$. We therefore split the series into two sums: an infinite one over $k > k_{**} := \lceil \log_\varphi (3 \bar t_\mathtt{hi}) \rceil = \lceil \log_\varphi (3(\tau + 1) \Delta^{-1} ) \rceil$ and a finite one over $k \in [1, k_{**}]$.
 
 The infinite sum can now be estimated by means of an integral over $[\log_\varphi (3(\tau+1)\Delta^{-1}), \infty)$ (which is a superset of $[(k_{**} + 1) - 1, \infty)$):
 \begin{align}
 \begin{split}
  \sum_{k = k_{**} + 1}^\infty \mathfrak{D}_2(q_k) &< \sum_{k = k_{**} + 1}^\infty \mathfrak{D}_2(\tilde F_k) \leqslant \int_{k_{**}}^\infty \mathfrak{D}_2 \left( {\varphi^x \over 3} \right) \, \integrald x \leqslant \int_{\log_\varphi (3(\tau+1)\Delta^{-1})}^\infty e^{-\varphi^x\Delta/3} \left( {\varphi^x \over 3} \right)^\tau \log \left( {\varphi^x \over 3} \right) \, \integrald x = \\
  &\stackrel{y = \varphi^x \Delta/3}{=} \int_{\tau + 1}^\infty e^{-y} y^{\tau - 1} \Delta^{-(\tau - 1)} (\log y + \log (\Delta^{-1})) \Delta^{-1} {\integrald y \over \log \varphi} \leqslant \\
  &\leqslant {1 \over 2} {\GammaEul(\tau) \over \log \varphi} \cdot \Delta^{-\tau} \log(\Delta^{-1}) + {\GammaEul'(\tau) \over \log \varphi} \cdot \Delta^{-\tau}.
 \end{split}
 \end{align}

 As for the finite sum we estimate it from above just as in the proof of lemma \ref{lem:dph1-estimates}: by the upper bound on its length multiplied by the upper bound on the value of the maximal summand:
 \begin{align}
 \begin{split}
  \sum_{k=1}^{k_{**}} \mathfrak{D}_2(q_k) &\leqslant \lceil \log_\varphi (3(\tau+1)\Delta^{-1}) \rceil \cdot \left( \max_{t \in [t_\mathtt{lo}, \bar t_\mathtt{hi}]} \mathfrak{D}_2(t) \right) \leqslant \\
  &\leqslant \left( \log_\varphi(3(\tau+1)\Delta^{-1}) + 1 \right) \cdot \left( \max_{t \in [t_\mathtt{lo}, \bar t_\mathtt{hi}]} \mathfrak{D}_1(t) \right) \cdot \left( \max_{t \in [t_\mathtt{lo}, \bar t_\mathtt{hi}]} \log t \right) = \\
  &= \log_\varphi(3\varphi(\tau+1)\Delta^{-1}) \cdot (\tau e^{-1})^\tau \Delta^{-\tau} \cdot \log ((\tau+1)\Delta^{-1}) = \\
  &= {(\tau e^{-1})^\tau \over \log \varphi} \Delta^{-\tau} \left( \log^2(\Delta^{-1}) + \log (3\varphi(\tau+1)^2) \cdot \log (\Delta^{-1}) + \log(3\varphi(\tau+1)) \log(\tau+1) \right).
 \end{split}
 \end{align}
 Adding the two estimates together yields the desired upper bound on $\Dph_2(\Delta)$.
\end{proof}

\begin{remark}
 We stress that in the case of $\Dph_1(\Delta) = O(\Delta^{-\tau} \log(\Delta^{-1}))$ the extra $\log (\Delta^{-1})$ factor above R\"ussmann's $O(\Delta^{-\tau})$ (\cite{russmann:1975}) stems from the very crude upper bound of the finite part of the sum by the $O(\log(\Delta^{-1}))$ length of the sum multiplied by the $O(\Delta^{-\tau})$ maximal summand. The sum in question, however, forms a rapidly incresing sequence and therefore perhaps the extra length factor can be dropped through a more clever estimate, as is the case with the simple
 \begin{equation}
  \sum_{k=0}^N 2^k = 2^{N+1} - 1 < 2^{N+1} = 2 \cdot 2^N = O(1) \cdot \max_{k \in \{ 0, \ldots, N \}} 2^k.
 \end{equation}
 
 In the case of $\Dph_2(\Delta) = O(\Delta^{-\tau} \log^2 (\Delta^{-1}))$ there are two extra $\log( \Delta^{-1} )$ factors - one of them seems to stem from the finite sum part as in the previous case, but the other one from the $\log t$ factor in $\mathfrak{D}_2(t)$. Again, it seems plausible to go down to $O(\Delta^{-\tau} \log(\Delta^{-1}))$, but dropping the second extra factor appears to be more delicate.  

\end{remark}

\subsection{Estimates for Khintchine-L\'evy numbers}\label{subsec:kl}

In this section we provide estimates for $\Brj_1(\Delta)$ and $\Brj_2(\Delta)$ whenever $\omega \in \KLB(T_-, T_+, N)$.

Again, as in subsection \ref{subsec:dioph}, it will be convenient to estimate the series by means of an integral, so we first introduce a general technical
\begin{lemma}[Series estimation by means of an improper integral for unimodal functions]\label{lem:series-estimate-by-integral-technical}
 Let $N \in \mathbb{N}$ and let $f : [N, \infty) \mapsto [0, \infty)$ be a continuous $L^1$ function. Assume that $f$ has a single local maximum at $x_0 \in [N, \infty)$. Then
 \begin{equation}
  \sum_{n=N}^\infty f(n) \leqslant f(x_0) + \int_N^\infty f(x) \, \integrald x.
 \end{equation}
\end{lemma}

\begin{proof}
 Since $f$ has a single local maximum $x_0$ we see that it increases on $[N, x_0]$ and decreases on $[x_0, \infty)$. In particular it increases on $[N, \lfloor x_0 \rfloor]$ and decreases on $[\lceil x_0 \rceil, \infty)$. If a function $g : \mathbb{R} \mapsto [0, \infty)$ is increasing on $[M, M+1]$ with $M \in \mathbb{R}$ then $g(M) \leqslant \int_M^{M+1} g(x) \, \integrald x$ and if it is decreasing we have $g(M+1) \leqslant \int_M^{M+1} g(x) \, \integrald x$.
 This way
 \begin{equation}
  \sum_{n=N}^{\lfloor x_0 \rfloor} f(n) \leqslant \int_N^{\lfloor x_0 \rfloor + 1} f(x) \, \integrald x \leqslant \int_N^{x_0} f(x) \, \integrald x + (\lfloor x_0 \rfloor + 1 - x_0) f(x_0)
 \end{equation}
 and
 \begin{equation}
  \sum_{n=\lceil x_0 \rceil}^\infty f(n) \leqslant \int_{\lceil x_0 \rceil - 1}^\infty f(x) \, \integrald x \leqslant \int_{x_0}^\infty f(x) \, \integrald x + (x_0 - (\lceil x_0 \rceil - 1)) f(x_0),
 \end{equation}
 where we majorized the ``excess'' part of the integral (i.e. the one on $[x_0, \lfloor x_0 \rfloor + 1]$ or $[\lceil x_0 \rceil - 1, x_0]$) by the length of the interval multiplied by the supremum of $f$.
 If $x_0 \not\in \mathbb{Z}$ then adding these two inequalities up gives
 \begin{equation}
  \sum_{n=N}^\infty f(n) \leqslant (2 - (\lceil x_0 \rceil - \lfloor x_0 \rfloor))f(x_0) + \int_N^\infty f(x) \, \integrald x = f(x_0) + \int_N^\infty f(x) \, \integrald x
 \end{equation}
 and if $x_0 \in \mathbb{Z}$ it gives
 \begin{equation}
  f(x_0) + \sum_{n=N}^\infty f(n) \leqslant (2 - (\lceil x_0 \rceil - \lfloor x_0 \rfloor))f(x_0) + \int_N^\infty f(x) \, \integrald x = 2f(x_0) + \int_N^\infty f(x) \, \integrald x,
 \end{equation}
 which concludes the proof.
\end{proof}

Before we proceed with the estimates of Brjuno-like functions we consider an ``idealized'' motivating example to provide some insight into the behavior of $\Brj_{1,2}$ for Khintchine-L\'evy $\omega$. According to the theorem on the Khintchine-L\'evy constant (\cite{levy:khintchine-levy}) for~a~``randomly chosen'' $\omega$ the asymptotic behavior of the sequence $q_n$ is exponential\footnote{More precisely the theorem tells us that for Lebesgue almost all $\omega$ we have $\root{n}\of{q_n} \to e^\ell$ as $n \to \infty$.}: $q_n \approx e^{\ell n}$ with $\ell = {\pi^2 \over 12 \log 2}$ being a constant independent of this randomly chosen $\omega$. We will thus first try to estimate $\Brj_1(\Delta)$ assuming that $q_n = e^{\ell n}$.

\begin{example}[Estimating the Brjuno-like functions in the ideal situation]\label{ex:ideal-qn-brjuno-estimate}
 Let $q_n = e^{\ell n}$ and $\Delta > 0$. We will estimate the size of $\Brj_1(\Delta)$ in terms of $\Delta$.
 \begin{align}
  e^{-\ell} \Brj_1(\Delta) &= e^{-\ell} \sum_{n=1}^\infty e^{-q_n \Delta} q_{n+1} = e^{-\ell} \sum_{n=1}^\infty e^{-e^{\ell n} \Delta} e^{\ell (n+1)} = \sum_{n=1}^\infty e^{\ell n - e^{\ell n} \Delta}. 
 \end{align}
 The function $f(x) = e^{\ell x - e^{\ell x} \Delta}$ has a single local maximum at $x_0 = \ell^{-1} \log \Delta^{-1}$ with $f(x_0) = e^{-1} \Delta^{-1}$. Lemma \ref{lem:series-estimate-by-integral-technical} tells us now that
 \begin{align}
  \sum_{n=1}^\infty e^{\ell n - e^{\ell n} \Delta} &\leqslant e^{-1} \Delta^{-1} + \int_1^\infty e^{\ell x - e^{\ell x} \Delta} \, \integrald x.
 \end{align}
 The integral can be computed by means of the substitution $y = e^{\ell x} \Delta$:
 \begin{equation}
  \int_1^\infty e^{\ell x - e^{\ell x} \Delta} \, \integrald x = \int_{e^\ell \Delta}^\infty (\ell \Delta)^{-1} e^{-y} \integrald y \leqslant \ell^{-1} \Delta^{-1} \int_0^\infty e^{-y} \integrald y = \ell^{-1} \Delta^{-1}.
 \end{equation}
 All in all
 \begin{equation}
  \Brj_1(\Delta) \leqslant G_\mathsf{Example} \cdot \Delta^{-1}
 \end{equation}
 with $G_\mathsf{Example} = e^\ell (e^{-1} + \ell^{-1}) \approx 3.9658$.
\end{example}

We see that for a geometric sequence $q_n$ one should expect $\Brj_1(\Delta)$ to be $O(\Delta^{-1})$ - this is somewhat the best case scenario in an attempt to estimate $\Brj_1(\Delta)$ for a truly ``randomly chosen'' $\omega$.

We will now consider the $\KL$-set case, more precisely the case of $\omega \in \KLB(T_-, T_+, N)$ for fixed parameters $T_-, T_+ > 0$ and $N \in \mathbb{N}$. For brevity we will write $\beta = \kappa-T_-, \beta' = \kappa'+T_+$ and $\gamma = \beta' / \beta$.
 
\begin{remark}
 The exponent in the $O \left( \Delta^{-1} \right)$ obtained in example \ref{ex:ideal-qn-brjuno-estimate} can be viewed as the negative relative growth rate of the upper bound of $q_n$ to its lower bound. In the ideal case of this example the relative growth rate is $1$, in the $\KL$ case it changes to $\gamma$. The results of the next lemma are consistent with this interpretation - the estimates we obtain are $O\left(\Delta^{-\gamma}\right)$ in the $\Brj_1(\Delta)$ case and $O\left( \Delta^{-\gamma} \log \left( \Delta^{-1} \right) \right)$ in the $\Brj_2(\Delta)$ case. One should therefore expect that a better sandwiching of $q_n$ will produce better estimates for the Brjuno-like functions. In our case the $\kappa$ and $\kappa'$ growth rates are somewhat the best possible, as they produce the measure estimates of the set $\KLB(T_-, T_+, N)$ in \cite{kamienski:2018-khintchine-levy}. The numerical value of the best possible $\gamma$ is $\kappa' / \kappa \approx 1.4278$, which is still close to $1$.
\end{remark}

We now have all the necessary ingredients to proceed with the
\begin{proof}[Proof of theorem \ref{thm:brjuno-estimates}]
 First observe that we can make initial estimates of the summands in the Brjuno-like functions by
 \begin{align}\label{eq:brjuno-summands-basic-estimate}
 \begin{split}
  e^{-q_n\Delta} q_{n+1} \leqslant e^{-M_n \Delta} M_{n+1}' \qquad \mbox{ and } \qquad  e^{-q_n\Delta} q_{n+1} \log a_{n+1} \leqslant e^{-M_n \Delta} M_{n+1}' \log {M_{n+1} \over M_n}
 \end{split}
 \end{align}
 for any $n \in \mathbb{N}$.
 \newcommand{\BrjFin}{\mathtt{BrjFinDiff}}
 Define, for a fixed $m \in \mathbb{N}$, the finite parts of $\Brj_1(\Delta)$ and $\Brj_2(\Delta)$ reduced by their ``expected upper estimates'' as
 \begin{align}\label{eq:brjfindiff}
 \begin{split}
  \BrjFin_1(m, \Delta) &= \sum_{n=1}^{m-1} e^{-q_n\Delta} q_{n+1} - \sum_{n=1}^{m-1} e^{-e^{\beta n} \Delta} \cdot e^{\beta' (n+1)} \\
  \BrjFin_2(m, \Delta) &= \sum_{n=1}^{m-1} e^{-q_n\Delta} q_{n+1} \log a_{n+1} - \sum_{n=1}^{m-1} e^{-e^{\beta n} \Delta} \cdot e^{\beta' (n+1)} \cdot \log \left( {e^{\beta' (n+1)} \over e^{\beta n} } \right).
 \end{split}
 \end{align}
 We have - under assumption that $\omega \in \KLB(T_-, T_+, N)$ - that
 \begin{align}
 \begin{split}
  \Brj_1(\Delta) &\leqslant \BrjFin_1(N, \Delta) + \sum_{n=1}^\infty e^{-e^{\beta n} \Delta} \cdot e^{\beta' (n+1)} = \BrjFin_1(N, \Delta) + e^{\beta'} \sum_{n=1}^\infty e^{\beta' n - e^{\beta n}\Delta}
 \end{split}
 \end{align}
 and similarly
 \begin{align}\label{eq:brj2-in-terms-of-sfrak1-and-sfrak2}
 \begin{split}
  \Brj_2(\Delta) &\leqslant \BrjFin_2(N, \Delta) + \sum_{n=1}^\infty e^{-e^{\beta n} \Delta} \cdot e^{\beta' (n+1)} \cdot \log \left( {e^{\beta' (n+1)} \over e^{\beta n} } \right) = \\
  &= \BrjFin_2(N, \Delta) + \sum_{n=1}^\infty e^{-e^{\beta n} \Delta} \cdot e^{\beta' (n+1)} \cdot ((T_+ + T_-)n + \beta') = \\
  &= \BrjFin_2(N, \Delta) + (T_+ + T_-)e^{\beta'} \sum_{n=1}^\infty e^{\beta' n - e^{\beta n}\Delta} n + e^{\beta' } \beta' \sum_{n=1}^\infty e^{\beta' n - e^{\beta n}\Delta}.
 \end{split}
 \end{align}
 
 We will denote the series appearing in the above estimates at the very end as $\Sfrak_1(\Delta)$ and $\Sfrak_2(\Delta)$:
 \begin{align}
 \begin{split}
  \Sfrak_1 (\Delta) = \sum_{n=1}^\infty e^{\beta' n - e^{\beta n}\Delta}, \qquad \Sfrak_2 (\Delta) = \sum_{n=1}^\infty e^{\beta' n - e^{\beta n}\Delta} n.
 \end{split}
 \end{align}
 The $O(1)$ parts in inequalities \eqref{eq:Brj-estimates-for-omega-KLBrj} are precisely $\BrjFin_1(N, \Delta)$ and $\BrjFin_2(N, \Delta)$. Note that they are not necessarily positive - whether they are or not depends on how much the first $N-1$ partial quotients of $\omega$ deviate from their average behavior.
 
 We will now estimate the series $\Sfrak_1(\Delta)$ and $\Sfrak_2(\Delta)$ by means of appropriate integrals. Define $B_1(x) = e^{\beta' x - e^{\beta x}\Delta}$ and $B_2(x) = x B_1(x)$. We will utilize lemma \ref{lem:series-estimate-by-integral-technical} for these two functions. The estimations follow similar paths to the ones performed for $\Dph_1(\Delta)$ and $\Dph_2(\Delta)$, outlined in the proofs of lemmas \ref{lem:dph1-estimates} and \ref{lem:dph2-estimates}, respectively.
 
 {\bf Estimating $\Sfrak_1(\Delta)$.} The function $B_1$ has a single maximum at $x_1 = \beta^{-1} \log \left(\gamma \Delta^{-1} \right)$ and the maximal value of $B_1$ is $B_1(x_1) = \left( \gamma e^{-1} \right)^\gamma \cdot \Delta^{-\gamma}$. We now compute the integral of $B_1$ over $[1, \infty)$ - this will, together with the value of $B_1(x_1)$, give us an estimate of $\Sfrak_1(\Delta)$ according to lemma \ref{lem:series-estimate-by-integral-technical}. We have
 \begin{equation}
  \int_1^\infty e^{\beta' x - e^{\beta x} \Delta} \, \integrald x = \int_{e^\beta \Delta}^\infty \beta^{-1} y^{-1} \cdot y^{\gamma} \Delta^{-\gamma} e^{-y} \, \integrald y \leqslant \beta^{-1} \Delta^{-\gamma} \int_0^\infty y^{\gamma - 1} e^{-y} \, \integrald y = \beta^{-1} \GammaEul(\gamma) \cdot \Delta^{-\gamma}
 \end{equation}
 after changing the variable to $y = e^{\beta x} \Delta$. All in all, according to lemma \ref{lem:series-estimate-by-integral-technical} we have
 \begin{equation}
  \Sfrak_1(\Delta) \leqslant \left( \beta^{-1} \GammaEul(\gamma) + \left(\gamma e^{-1} \right)^\gamma  \right) \cdot \Delta^{-\gamma}
 \end{equation}
 which gives the first inequality in \eqref{eq:Brj-estimates-for-omega-KLBrj}.

 {\bf Estimating $\Sfrak_2(\Delta)$.} First observe that for small enough $\Delta$ the function $B_2$ also has a unique local maximum $x_2$. It is, however, not possible to give it by means of an explicit formula as was the case with $B_1$, but we can infer that the maximum is located in the vicinity of $x_1$, namely somewhere in the interval
 \newcommand{\xlo}{x_\mathtt{lo}}
 \newcommand{\xhi}{x_\mathtt{hi}}
 \begin{equation}
  (\xlo, \xhi) = \left( x_1, x_1 + {1 \over \beta} \log \left( 1 + {1 \over \gamma \log \left( \gamma \Delta^{-1} \right) } \right) \right).
 \end{equation}
 To see this observe that $B_2'(\xlo) > 0$ while $B_2'(\xhi) < 0$ and the sign of $B_2'(x)$ changes exactly once from positive to negative as $x$ grows. This is a narrow constraint, since the length of the interval is small: $\xhi - \xlo = O \left( 1 / \log \left( \Delta^{-1} \right) \right)$. We now estimate the value of $B_2(x_2)$ - to do this we will provide an upper bound for $B_2(x)$ over all $x \in (\xlo, \xhi)$. We have $B_2(x) = xB_1(x)$, therefore - for sufficiently small $\Delta$ (in this case $\Delta \leqslant \gamma e^{-1/\gamma}$) - we have
 \begin{align}
 \begin{split}
  \sup_{x \in (\xlo, \xhi)} B_2(x) &= \sup_{x \in (\xlo, \xhi)} xB_1(x) \leqslant \left( \sup_{x \in (\xlo, \xhi)} x \right) \cdot \left( \sup_{x \in (\xlo, \xhi)} B_1(x) \right) = \xhi \cdot B_1(x_1) = \\
  &= \left( {1 \over \beta} \log \left( \gamma \Delta^{-1} \right) + {1 \over \beta} \log \left( 1 + { 1 \over \gamma \log \left( \gamma \Delta^{-1} \right) } \right) \right) \left( \gamma e^{-1} \right)^\gamma \cdot \Delta^{-\gamma} \leqslant \\
  &\leqslant {1 \over \beta} \left( \gamma e^{-1} \right)^\gamma \cdot \Delta^{-\gamma} \cdot \log \left( 2\gamma \Delta^{-1} \right) = {1 \over \beta} \left( \gamma e^{-1} \right)^\gamma \log (2\gamma) \cdot \Delta^{-\gamma} + \left( \gamma e^{-1} \right)^\gamma \cdot \Delta^{-\gamma} \log \left( \Delta^{-1} \right).
 \end{split}
 \end{align}

 We now turn to the integral of $B_2$ over $[1, \infty)$. We will once more utilize the substitution $y = e^{\beta x} \Delta$.
 
 \begin{align}
 \begin{split}
  \int_1^\infty B_2(x) \, \integrald x &= \int_1^\infty x e^{\beta' x - e^{\beta x} \Delta} \, \integrald x = \int_{e^\beta \Delta}^\infty {1 \over \beta} \left( \log y + \log \left( \Delta^{-1} \right) \right) \cdot y^\gamma \Delta^{-\gamma} \cdot e^{-y} \cdot {\integrald y \over \beta y} \leqslant \\
  &\leqslant {1 \over \beta^2} \cdot \Delta^{-\gamma} \cdot \left( \int_0^\infty e^{-y} y^{\gamma - 1} \log y \, \integrald y + \log \left( \Delta^{-1} \right) \cdot \int_0^\infty e^{-y} y^{\gamma - 1} \, \integrald y \right) =\\
  &= {1 \over \beta^2} \GammaEul'(\gamma) \cdot \Delta^{-\gamma} + {1 \over \beta^2} \GammaEul(\gamma) \cdot \Delta^{-\gamma} \log \left( \Delta^{-1} \right).
 \end{split}
 \end{align}
 
 We now have both ingredients necessary to use lemma \ref{lem:series-estimate-by-integral-technical}, which yields
 \begin{align}
 \begin{split}
  \Sfrak_2(\Delta) &\leqslant \left( {1 \over \beta} \left( \gamma e^{-1} \right)^\gamma \log (2\gamma) + {1 \over \beta^2} \GammaEul'(\gamma) \right) \cdot \Delta^{-\gamma} + \left( \left( \gamma e^{-1} \right)^\gamma + {1 \over \beta^2} \GammaEul(\gamma) \right) \cdot \Delta^{-\gamma} \log \left( \Delta^{-1} \right).
 \end{split}
 \end{align}
 Combining this together with \eqref{eq:brj2-in-terms-of-sfrak1-and-sfrak2} gives the second inequality in \eqref{eq:Brj-estimates-for-omega-KLBrj}.
 
\end{proof}

\section{Brjuno numbers vs. semi-Brjuno numbers}\label{sec:brjuno-semibrjuno}

In this section we briefly discuss the concept of semi-Brjuno numbers. In brief: Brjuno numbers are semi-Brjuno, but not the other way round.

\begin{lemma}[Brjuno numbers are semi-Brjuno]\label{lem:brjuno-are-semibrjuno}
 Suppose $\omega \in \Brjuno$. Then $\omega \in \SemiBrjuno$.
\end{lemma}

\begin{proof}
 Denote $b_n := q_n^{-1} \log q_{n+1}$, $B_n^{(1)}(\Delta) := e^{-q_n \Delta} q_{n+1}$ and $B_n^{(2)}(2\Delta) := e^{-2q_n \Delta} q_{n+1} \log a_{n+1}$ and fix $\Delta > 0$. We will prove that $\sum_{n=1}^\infty B_n^{(j)}(j\Delta)$ converges for $j \in \{ 1, 2 \}$. For Brjuno numbers $\sum_{n=1}^\infty b_n$ converges, therefore in particular $b_n \to 0$. Expressing $q_{n+1}$ explicitly in terms of $b_n$ and $q_n$ gives $q_{n+1} = e^{b_n q_n}$. Using this and the (crude over-)estimate $\log a_{n+1} < q_{n+1}$ we can write 
 \begin{equation}
  B_n^{(1)}(\Delta) \leqslant e^{-q_n \Delta} q_{n+1} \leqslant e^{-q_n(\Delta - b_n)} \qquad \mbox{and} \qquad B_n^{(2)}(2\Delta) \leqslant e^{-2q_n \Delta} q_{n+1}^2 \leqslant e^{-2q_n(\Delta - b_n)}.
 \end{equation}
 For $n$ larger than some $N(\Delta)$ we have $b_n < \Delta/2$ and this way $B_n^{(j)}(j\Delta) \leqslant e^{-j q_n \Delta/2} < e^{-j n\Delta/2}$ for such $n$ and $j \in \{ 1, 2 \}$, which confirms the convergence of $\sum_{n=N(\Delta)}^\infty B_n^{(j)}(j\Delta)$.
\end{proof}

Note how in the proof of lemma \ref{lem:brjuno-are-semibrjuno} we only use the fact that $b_n \to 0$ and not the full assumption of $\sum_{n=1}^\infty b_n$ being convergent. This observation will allow us to construct a semi-Brjuno number which is not Brjuno.

\begin{lemma}
 There exists $\omega^* \in \SemiBrjuno \setminus \Brjuno$.
\end{lemma}

\begin{proof}
 We will define $\omega^*$ through its partial quotients $(a_n)$, setting them recursively in terms of the previous denominators of convergents:
 \begin{equation}
  a_{n+1} := \left\lfloor {1 \over q_n} \exp (q_n \cdot \alpha_n) \right\rfloor - 1
 \end{equation}
 where $(\alpha_n)$ is any sequence for which $\alpha_n > 0$ and $\alpha_n \to 0$, but $\sum_{n=1}^\infty \alpha_n$ diverges. The first partial quotient $a_1$ can be chosen arbitrarily. This gives recursive upper and lower bounds on the growth of the sequence $(q_n)$, valid for $n$ large enough:
 \begin{equation}
  q_{n+1} < (1+a_{n+1}) q_n < q_n \left\lfloor {1 \over q_n} \exp (q_n \cdot \alpha_n) \right\rfloor < \exp (q_n \cdot \alpha_n)
 \end{equation}
 and
 \begin{equation}
  q_{n+1} > a_{n+1} q_n = \left( \left\lfloor {1 \over q_n} \exp (q_n \cdot \alpha_n) \right\rfloor - 1 \right) q_n > {1 \over 2} \exp (q_n \cdot \alpha_n)
 \end{equation}
 (we used the inequality $\lfloor y \rfloor - 1 > y/2$, valid for large enough $y$).
 
 We first show that $\omega^* \in \SemiBrjuno$. Fix $\Delta > 0$ and observe that for $n$ large enough we have $\alpha_n < \Delta/2$ (since $\alpha_n \to 0$) and thus
 \begin{equation}
   e^{-q_n \Delta} q_{n+1} \leqslant e^{-q_n\left(\Delta - \alpha_n \right)} \leqslant e^{-q_n\left(\Delta - {\Delta \over 2}\right)} = e^{-q_n \Delta / 2} < e^{-n \Delta / 2}.
 \end{equation}
 The final upper bound gives rise to a convergent series. Similarily
 \begin{equation}
   e^{-2q_n \Delta} q_{n+1} \log a_{n+1} \leqslant e^{-2q_n \Delta} q_{n+1}^2 \leqslant e^{-2q_n\left(\Delta - \alpha_n \right)} \leqslant e^{-2q_n\left(\Delta - {\Delta \over 2}\right)} = e^{-q_n \Delta} < e^{-n \Delta}.
 \end{equation}
 
 To see that $\omega^* \not\in \Brjuno$ observe that
 \begin{equation}
  {\log q_{n+1} \over q_n} \geqslant \alpha_n - {\log 2 \over q_n}.
 \end{equation}
 This lower bound gives rise to a divergent series, since the term $\log 2/q_n$ gives rise to a convergent one as the growth of $(q_n)$ is at least exponential.
\end{proof}

\section{A counterexample}\label{sec:counterexample}

In the previous sections we learnt that the convergence of $\Brj_1(\Delta)$ and $\Brj_2(2\Delta)$ with $\Delta = (1+\omega)\delta$ is sufficient for the existence of an analytic solution $g \in \mathcal{P}_0(\varrho - \delta)$ to the cohomological equation with initial data $a \in \mathcal{P}_0(\varrho)$. We now show that when $\Brj_1(\Delta)$ or $\Brj_2(2\Delta)$ do not converge, then one can construct a function $\hat a \in \mathcal{P}_0(\varrho)$ with $||\hat a||_\varrho$ arbitrarily small, for which $\hat g$, the formal solution of the cohomological equation with initial data $\hat a$, is not analytic in $\Pi(\varrho- \delta')^2$ for any $0 < \delta' < \delta$. In particular if $\omega \not \in \SemiBrjuno$ then the cohomological equation may not be solvable in $\mathcal{P}_0(\varrho - \delta)$ even for arbitrarily small $\delta \in (0, \varrho)$.

We define $\hat a$ through its Fourier coefficients $\hat a_{p,q}, (p,q) \in \mathbb{Z}^2_*$. The idea is to maximize the crucial ones, while keeping their growth tame, so that $\hat a$ falls within the desired analyticity class. Before we proceed we introduce several necessary quantities.

Let $\varepsilon > 0$ and let $(\alpha_n)_{n=1}^\infty \subset \mathbb{R}$ be a sequence given by $\alpha_n := 1/(2\bar \alpha q_n)$, where $\bar \alpha := \sum_{n=1}^\infty 1/q_n$ (this way $2 \sum_{n=1}^\infty \alpha_n = 1$). We now put
\begin{equation}\label{eq:apq-counterexample}
 \hat a_{p,q} = \begin{cases} \varepsilon e^{-\varrho(|p| + |q|)} \cdot \alpha_n &\mbox{ whenever } (p,q) = \pm(p_n, q_n) \mbox{ for some } n\geqslant 1 \\ 0 &\mbox{ for all remaining } (p,q) \in \mathbb{Z}^2_* \end{cases}.
\end{equation}

\begin{lemma}
 The function $\hat a$ defined through \eqref{eq:apq-counterexample} satisfies $\hat a \in \mathcal{P}_0(\varrho)$ and $||\hat a||_\varrho \leqslant \varepsilon$.
\end{lemma}

\begin{proof}
 We first show that $\hat a$ is analytic in $\Pi(\varrho)^2$. Indeed, it is defined as a limit of analytic functions (partial sums of the Fourier expansion), we therefore only need to verify that the defining Fourier series converges uniformly on compact subsets of $\Pi(\varrho)^2$. To see this first observe that each compact $K \subset \Pi(\varrho)^2$ is actually a subset of $\Pi(\varrho-\theta)^2$ for some $\theta > 0$. This way the norm of $e_{p,q}$ on $K$ is at most $e^{(\varrho - \theta)(|p|+|q|)}$ and thus for all $(\tilde x, \tilde y) \in K$ we have
 \begin{equation}
  |\hat a_{p,q} e_{p,q}(\tilde x, \tilde y)| \leqslant \varepsilon e^{-\varrho(|p| + |q|)} \cdot A \cdot e^{(\varrho - \theta)(|p|+|q|)} = A \varepsilon e^{-\theta(|p|+|q|)},
 \end{equation}
 where $A := \max_{n \in \mathbb{N}} \alpha_n$. The last estimate gives a convergent series when summed over all $(p,q) \in \mathbb{Z}^2_*$, therefore by the Weierstrass M-test the series is uniformly convergent on $K$.
 
 To see that to analyticity we can add continuity up to the boundary of the domain note that each summand $\hat a_{p,q} e_{p,q}$ in the Fourier series of $\hat a$ is a continuous function on $\overline{\Pi(\varrho)^2}$. The sup-norm on $\overline{\Pi(\varrho)^2}$ of each summand is either $0$ or bounded by $\varepsilon \alpha_n$, therefore the Fourier series converges uniformly also on $\overline{\Pi(\varrho)^2}$ again by the Weierstrass M-test as $2\sum_{n=1}^\infty \varepsilon \alpha_n < \infty$. The function $\hat a$ is now continuous on $\overline{\Pi(\varrho)^2}$ as a uniform limit of continuous functions.
 
 The coefficients $\hat a_{p,q}$ satisfy the Hermitian symmetry condition $\hat a_{p,q} = \overline{\hat a_{-p, -q}}$, therefore $\hat a$ returns real values for real pairs of arguments.
 
 The absolute value of $\hat a_{p,q}$ is either $0$ or satisfies $|\hat a_{\pm p_n, \pm q_n}| = \varepsilon e^{-\varrho(p_n + q_n)} \alpha_n$, therefore we can estimate $||\hat a||_\varrho$ from above by
 \begin{equation}
  ||\hat a||_\varrho \leqslant \sum_{(p,q) \in \mathbb{Z}^2_*} |\hat a_{p,q}| \cdot ||e_{p,q}||_\varrho \leqslant 2 \sum_{n=1}^\infty \varepsilon e^{-\varrho(p_n+q_n)} \alpha_n \cdot e^{\varrho(p_n+q_n)} = \varepsilon.
 \end{equation}
\end{proof}

\begin{lemma}\label{lem:lack-of-analyticity}
 Let $\varrho >0, \varepsilon > 0, \omega \in \mathbb{X}, 0 < \delta' < \delta < \varrho$. Denote $\Delta = (1+\omega)\delta$ and suppose $\Brj_1(\Delta) = \infty$ or $\Brj_2(2\Delta) = \infty$. Then $\hat g$, the formal solution of the cohomological equation with initial data $\hat a$, does not belong to $\mathcal{P}_0(\varrho - \delta')$.
\end{lemma}

Before we proceed with the proof we present two simple lemmas.
\begin{lemma}[A divergent series majorizes a convergent one on a subsequence]\label{lem:divergent>convergent}
 Suppose $(c_n)_{n=1}^\infty$ and $(d_n)_{n=1}^\infty$ are sequences of positive numbers such that $\sum_{n=1}^\infty c_n$ is convergent and $\sum_{n=1}^\infty d_n$ is divergent. Then $d_n > c_n$ for infinitely many $n$.
\end{lemma}

\begin{proof}
 Suppose the contrary - that there are at most finitely many indices $n$ for which $d_n > c_n$. Denote by $N$ the maximal one and note that for $n > N$ we have $d_n \leqslant c_n$ which is a contradiction, since all tails of a divergent series are also divergent series.
\end{proof}

\begin{lemma}[An analyticity criterion]\label{lem:analyticity-criterion}
 Let $R > 0$. If there exists an infinite sequence of indices $(\nu_k, \mu_k)_{k=1}^\infty \subset \mathbb{Z}^2_*$ such that the Fourier coefficients of a function $b : \Pi(R)^2 \mapsto \mathbb{C}$ satisfy $e^{R(|\nu_k|+|\mu_k|)} |b_{\nu_k,\mu_k}| \to \infty$ as $k \to \infty$ then $b \not \in \mathcal{P}_0(R)$.
\end{lemma}

\begin{proof}
 The statement of the lemma is a direct contraposition of lemma \ref{lem:exponential-decay-of-fourier-terms}.
\end{proof}

\begin{proof}[Proof of lemma \ref{lem:lack-of-analyticity}]
 Denote $\Delta' = (1+\omega)\delta'$.
 
 {\bf Case 1: $\Brj_1(\Delta) = \infty$.}
 
 A direct computation shows that
 \begin{equation}\label{eq:lack-of-analyticity-estimate-1}
  e^{(\varrho - \delta')(p_n + q_n)} |\hat g_{p_n,q_n}| = \varepsilon {e^{-\delta'(p_n + q_n)} \over |q_n \omega - p_n|} \alpha_n > \varepsilon e^{-\delta'} e^{-\Delta' q_n} q_{n+1} \alpha_n = \varepsilon e^{-\delta'} \cdot {1 \over \bar \alpha q_n} \cdot e^{-\Delta q_n} q_{n+1} \cdot e^{(\Delta - \Delta') q_n}.
 \end{equation}
 Since we assumed that $\Brj_1(\Delta)$ diverges we have, by virtue of lemma \ref{lem:divergent>convergent}, that for infinitely many $n$'s the quantity $e^{-q_n \Delta} q_{n+1}$ majorizes one that gives a convergent series, say $1/q_n$. For such $n$ we can thus further estimate
 \begin{equation}
  e^{(\varrho - \delta')(p_n + q_n)} |\hat g_{p_n,q_n}| \geqslant \ldots \geqslant {\varepsilon e^{-\delta'} \over \bar \alpha} \cdot {1 \over q_n^2} \cdot e^{(\Delta - \Delta')q_n}.
 \end{equation}
 The last expression can be made arbitrarily large since $\Delta - \Delta' > 0$ and the exponential terms overcomes the $\mbox{const.}/q_n^2$ term, thus, by lemma \ref{lem:analyticity-criterion} the function $\hat g$ cannot be analytic in $\Pi(\varrho - \delta')^2$.
 
 {\bf Case 2: $\Brj_2(2\Delta) = \infty$.}
 
 Similarly to \eqref{eq:lack-of-analyticity-estimate-1} we have
 \begin{align}
 \begin{split}
  e^{(\varrho - \delta')(p_n + q_n)} |\hat g_{p_n,q_n}| &> \varepsilon e^{-\delta'} \cdot {1 \over \bar \alpha q_n} \cdot e^{-\Delta q_n} q_{n+1} \cdot e^{(\Delta - \Delta') q_n} = \varepsilon e^{-\delta'} \cdot {1 \over \bar \alpha q_n} \cdot \left[ e^{-2\Delta q_n} q_{n+1}^2 \right]^{1/2} \cdot e^{(\Delta - \Delta') q_n} > \\
  &> \varepsilon e^{-\delta'} \cdot {1 \over \bar \alpha q_n} \cdot \left[ e^{-2\Delta q_n} q_{n+1} \log a_{n+1} \right]^{1/2} \cdot e^{(\Delta - \Delta') q_n}.
 \end{split}
 \end{align}
 By lemma \ref{lem:divergent>convergent} the contents of the square brackets can be estimated from below for infinitely many $n$'s by a sequence whose series is convergent, say $1/q_n$. The remaining part of the proof is analogous to case 1.
\end{proof}

\begin{remark}
 Note that in lemma \ref{lem:lack-of-analyticity} the condition $\Brj_2(2\Delta) = \infty$ can be substituted by $\Brj_2(\Delta^*) = \infty$ for some $\Delta^* > \Delta$ and the proof will also follow. Also the two sequences whose series are convergent that we chose in the construction of $\hat a$ and in the proof, $(\alpha_n)$ and $(1/q_n)$, can be made much more slowly convergent to $0$, e.g. both equal to $n^{-(1+\beta)}$ with $\beta > 0$ small. The striking difference between the inverse polynomial speed of their product $n^{-2(1+\beta)}$ and the (at least) doubly exponential $e^{(\Delta - \Delta')q_n}$ shows how little analyticity is actually lost while solving the cohomological equation.
\end{remark}

\bibliographystyle{plain}
\bibliography{bibliography.bib}

\end{document}